\documentclass[aps,pre,reprint,longbibliography,superscriptaddress,noshowpacs,floatfix]{revtex4-2}
\usepackage{amsmath,amssymb,amsthm,mathrsfs,}
\usepackage{graphicx}
\usepackage{tikz-cd}
\usepackage{bm}
\usepackage{bbm}
\usepackage{comment}
\usepackage[colorlinks=true, allcolors=blue]{hyperref}
\usepackage{ragged2e}
\usepackage{enumitem}

\newtheorem{theorem}{Theorem}
   
   \newtheorem{lemma}{Lemma}
   \newtheorem*{lemma*}{Lemma}

\theoremstyle{definition}
   
   \newtheorem{example}[theorem]{Example}
\theoremstyle{remark}    
  \newtheorem{remark}[theorem]{Remark}

\newcommand*{\sbar}{\skew{3}{\bar}}
\newcommand*{\sdot}{\skew{3}{\dot}}
\DeclareMathAlphabet{\mathpzc}{OT1}{pzc}{m}{it}
\DeclareMathOperator{\tr}{tr}

\DeclareMathOperator{\diag}{diag}

\DeclareMathOperator{\Imag}{Im}
\DeclareMathOperator{\Arg}{Arg}

\DeclareMathOperator{\SL}{SL}
\DeclareMathOperator{\PSL}{PSL}

\DeclareMathOperator{\SU}{SU}
\DeclareMathOperator{\PSU}{PSU}
\DeclareMathOperator{\Aut}{Aut}

\newcommand{\dif}{\mathop{}\!\mathrm{d}}

\makeatletter
\newcommand{\spx}[1]{%
  \if\relax\detokenize{#1}\relax
    \expandafter\@gobble
  \else
    \expandafter\@firstofone
  \fi
  {^{#1}}%
}
\makeatother

\newcommand\pd[3][]{\frac{\partial\spx{#1}#2}{\partial#3\spx{#1}}}

\newcommand{\genericdel}[4]{%
  \ifcase#3\relax
  \ifx#1.\else#1\fi#4\ifx#2.\else#2\fi\or
  \bigl#1#4\bigr#2\or
  \Bigl#1#4\Bigr#2\or
  \biggl#1#4\biggr#2\or
  \Biggl#1#4\Biggr#2\else
  \left#1#4\right#2\fi
}

\newcommand{\sVert}[1][0]{%
  \ifcase#1\relax
  \rvert\or\bigr|\or\Bigr|\or\biggr|\or\Biggr
  \fi
}

\begin{document}

% Title, authors and affiliations
\title{Operator-theoretic approach to the partial integration\\ of randomly coupled phase oscillators}

\author{Vincent Thibeault}\email[]{vincent.thibeault.1@ulaval.ca}
\affiliation{D\'epartement de physique, de g\'enie physique et d'optique, Universit\'e Laval, Qu\'ebec (Qc), Canada}% G1V 0A6}
\affiliation{Centre interdisciplinaire en mod\'elisation math\'ematique de l'Universit\'e Laval, Qu\'ebec (Qc), Canada} % G1V 0A6}

\author{Benjamin Claveau}%\email[]{benjamin.claveau.1@ulaval.ca}
\affiliation{D\'epartement de physique, de g\'enie physique et d'optique, Universit\'e Laval, Qu\'ebec (Qc), Canada}% G1V 0A6}
\affiliation{Centre interdisciplinaire en mod\'elisation math\'ematique de l'Universit\'e Laval, Qu\'ebec (Qc), Canada} % G1V 0A6}

\author{Antoine Allard}% \email[]{antoine.allard@phy.ulaval.ca}
\affiliation{D\'epartement de physique, de g\'enie physique et d'optique, Universit\'e Laval, Qu\'ebec (Qc), Canada}% G1V 0A6}
\affiliation{Centre interdisciplinaire en mod\'elisation math\'ematique de l'Universit\'e Laval, Qu\'ebec (Qc), Canada} % G1V 0A6}

\author{Patrick Desrosiers}%\email[]{patrick.desrosiers@phy.ulaval.ca}
\affiliation{D\'epartement de physique, de g\'enie physique et d'optique, Universit\'e Laval, Qu\'ebec (Qc), Canada} % G1V 0A6}
\affiliation{Centre interdisciplinaire en mod\'elisation math\'ematique de l'Universit\'e Laval, Qu\'ebec (Qc), Canada}% G1V 0A6}
\affiliation{Centre de recherche CERVO, Qu\'ebec (Qc), Canada}% G1E 0N5

\begin{abstract}

In our previous work~\cite{Thibeault2026}, we adopted Koopman theory to link the existence of different constants of motion to the presence of specific network motifs of Kuramoto oscillators. Yet, it remains to be shown how the partial integration can be carried out using the Koopman generator and its eigenfunctions. In this paper, we construct a random graph from network motifs that admit Koopman eigenfunctions and conserved quantities, and use it to define a partially integrable Kuramoto model. We perform the partial integration of the introduced model when there are monomial eigenfunctions and conserved cross-ratios, while providing an operator-theoretic derivation of the Watanabe-Strogatz transformation based on Magnus expansion and a recent result on closed forms of the Baker-Campbell-Hausdorff formula~\cite{Matone2015}. % We conclude with an example of the partial integration. % leads to only three equations. 
\end{abstract}

%\pacs{64.60.aq}
\maketitle

\let\oldaddcontentsline\addcontentsline% Store \addcontentsline
\renewcommand{\addcontentsline}[3]{}% Make \addcontentsline a no-op

\twocolumngrid

\section*{Introduction}

From pacemaker cells and neurons to fireflies and orchestras, nature’s constituents often tend
to synchronize over some period of time~\cite{Winfree2000, Pikovsky2003, Strogatz2003, Izhikevich2007, Boccaletti2018}. The construction of mathematical models has been an important step in finding rigorous insights on synchronization over the years, notably with the work of Winfree~\cite{Winfree1967}, Kuramoto~\cite{Kuramoto1975}, Ermentrout and Kopell~\cite{Ermentrout1986}. These models are nonlinear, and synchronization emerges as one of the most remarkable
collective phenomena they can describe. They exhibit a great richness of oscillatory behaviors and attempts to find analytical solutions for the differential equations describing the models are generally in vain. But there are exceptions.

In 1993 and 1994, Watanabe and Strogatz made pioneering developments in the (partial) integrability of identical phase oscillators with the form $\dot\theta_j = f(\bm{\theta}) + g(\bm{\theta})\cos\theta_j + h(\bm{\theta})\sin\theta_j$, where $j\in\{1,...,N\}$, $\theta_j(t)$ is the $j$-th oscillator phase at time $t$, $\dot\theta_j$ is its derivative with respect to the time $t$, $\bm{\theta} = (\theta_1,...,\theta_N)$, and $f, g, h$ are $2\pi$-periodic functions in each of their argument~\cite{Watanabe1993, Watanabe1994}. They found an astute transformation---generally called the Watanabe-Strogatz (WS) transformation---that allows reducing the number of equations to 3 with $N - 3$ constants of motion, along with conditions leading to complete integrability. The transformation is defined as
\begin{align}\label{eq:ws}
    \tan[{\textstyle{\frac{1}{2}}}(\theta_j - \Theta)] = \sqrt{\frac{1 + \gamma}{1 - \gamma}}\,\tan[{\textstyle{\frac{1}{2}}}(\psi_j - \Psi)]\,, %\tan[{\textstyle{\frac{1}{2}}}(\theta_j(t) - \Theta(t))] = \sqrt{\frac{1 + \gamma(t)}{1 - \gamma(t)}}\,\tan[{\textstyle{\frac{1}{2}}}(\psi_j - \Psi(t))]\,, 
\end{align}
where $\psi_1,...,\psi_N$ are constants and the functions of time $\gamma, \Theta, \Psi$ are determined by solving their differential equations, obtained by applying the coordinate change~\eqref{eq:ws} to the phase dynamics. 

In response to these findings, Goebel observed that the introduced class of identical oscillators is in fact linked to a particular kind of coupled ``Riccati equations'' and that the WS transformation is ultimately related to a M\"{o}bius transformation~\cite{Goebel1995}. It was then clarified that the constants of motion are cross-ratios and that the WS transformation is equivalent to a disk automorphism~\cite{Marvel2009}. This inspired researchers to use projective geometry to perform the partial integration~\cite{Stewart2011}, to make connections with hyperbolic geometry~\cite{Chen2017b, Chen2019, Lipton2021}, and to broaden the applicability of WS theory (e.g., the work of Pikovsky \& colleagues~\cite{Pikovsky2008, Pikovsky2011, Vlasov2016a} and Lohe~\cite{Lohe2017, Lohe2018, Lohe2019, Lohe2020, Lohe2025, Lohe2025f}). However, open questions remain regarding the partial integrability of phase oscillator dynamics, which is still an active research area~\cite{Chen2017b, Lipton2021, Thumler2023, Cestnik2024}.

For phase dynamics with non-identical frequencies, noise, and on heterogeneous graphs, exact partial integration is generally not possible. This has motivated the development of a wide range of approximate dimension-reduction methods. In the infinite-size limit, the dominant approach is to use the Ott-Antonsen Ansatz or its variants~\cite{Ott2008, Abrams2008, Chen2017, Bick2020, Engelbrecht2020, Cestnik2022b, Cestnik2022c}. For noisy oscillators, methods include the use of cumulants~\cite{Tyulkina2018,Cestnik2022b, Goldobin2018, Goldobin2019,Goldobin2019oa, Goldobin2021} and the recent mesoscopic theory of Buendía~\cite{Buendia2025}. Collective coordinates \cite{Gottwald2015, Gottwald2017} and different spectral methods~\cite{Gfeller2007, Gfeller2008, Kalloniatis2010, Thibeault2020, Timofeyev2025} have also been developed in the finite-size limit.

Less attention has been devoted to the partial integrability of phase oscillators with heterogeneous connections. It is known that WS transformations can be applied to all-to-all coupled communities~\cite{Pikovsky2008, Hong2011_pre} and to peripheries of star graphs~\cite{Vlasov2015, Vlasov2015a, Vlasov2017, Xu2018, Xu2019, Chen2019a}, but other connectivity patterns can also enable partial integrability.

As we showed in Ref.~\cite{Thibeault2026} through Koopman theory~\cite{Koopman1931, Koopman1932, VonNeumann1932, Carleman1932, Budisic2012, Brunton2022, Mezic2004, Mezic2005, Mauroy2013, Mauroy2016, Mezic2019, Joseph2020}, various motifs allow constants of motion to exist for the Kuramoto model and these motifs occur in different empirical complex networks. From this perspective, different network motifs can not only exhibit distinct stability and synchronizability properties, as extensively shown in the literature~\cite{Ashwin1992, Golubitsky1999, *Golubitsky2002, *Golubitsky2006, *Golubitsky2023, Milo2002, Moreno2004, Lodato2007, DHuys2008, Morone2019a, *Morone2020, *Makse2025, Angulo2015, Schaub2016, Aguiar2018}, but can also possess distinct ``integrability properties''.

Yet, we did not show how to partially integrate the Kuramoto model within this Koopman framework. In this paper, we introduce a random graph based on motifs admitting different constants of motion, thereby obtaining a partially integrable Kuramoto model. We then show how to deduce the transformations (including WS transformations~\eqref{eq:ws}) to perform the partial integration of the introduced model. To this end, we use operator-theoretic tools, from monomial (Koopman) eigenfunctions to Magnus expansion and a closed-form Baker-Campbell-Hausdorff (BCH) formula recently introduced by Matone~\cite{Matone2015}. 

In Sec.~\ref{sec:different_descriptions}, we present different descriptions of the Kuramoto model, especially to introduce its operator-theoretic description with the Koopman generator. Then, we construct the partially integrable model, as summarized in Table~\ref{tab:random_graph}, by imposing the constraints that guarantee the presence of monomial eigenfunctions and conserved cross-ratios. Finally, we partially integrate the model by first treating the parts that admit monomial eigenfunctions [Sec.~\ref{subsec:partial_integration_monomial}] and then those that support conserved cross-ratios [Sec.~\ref{subsec:partial_integration_crossratio}]. The reduced model is finally presented in Table~\ref{tab:partially_integrated}.

\section{Operator-theoretic description of the Kuramoto model}
\label{sec:different_descriptions}

The typical description of the (generalized) Kuramoto model is 
\begin{equation}\label{eq:kuramoto}
    \dot{\theta}_j = \omega_j + \sum_{k \in\mathcal{V}} W_{jk} \sin(\theta_k - \theta_j - \alpha_{jk})\,,\quad j\in\mathcal{V}\,,
\end{equation}
where $\mathcal{V} = \{1,...,N\}$, $\omega_j\in\mathbb{R}$ is the $j$-th natural frequency, $W$ is a weight matrix where $W_{jk}\in\mathbb{R}$ is the weight of the connection from $k$ to $j$, and $-\pi/2 < \alpha_{jk}\leq \pi/2$ is the $(j,k)$-th element of the phase-lag matrix $\alpha$~\cite{Kuramoto1975, Sakaguchi1986, Rodrigues2016}. Without loss of generality, we set $W_{jj} = 0$ and $\alpha_{jj} = 0$ for all $j\in\{1,...,N\}$.

Setting $z_j = e^{i\theta_j}$ for all $j$ yields a more compact description of the above Kuramoto model, that is,
\begin{equation}\label{eq:kuramotoz}
    \dot{z}_j = p_j(\bm{z}) - \overline{p_j(\bm{z})}z_j^2\,,\quad p_j(\bm{z}) = \sum_{k \in\mathcal{V}} A_{jk}z_k\,,% \sum_{k \in\mathcal{V}} (A_{jk}z_k - \bar{A}_{jk}\bar{z}_kz_j^2)\,,   
\end{equation}
where, given $\bm{\omega} = (\omega_1,...,\omega_N)$ and $e^{-i\alpha} = (e^{-i\alpha_{jk}})_{j,k\in\mathcal{V}}$,
\begin{align}\label{eq:complex_matrix}
     A = \frac{1}{2}\left( W \circ e^{-i\alpha} + i\diag(\bm{\omega})\right)\,
\end{align}
describes the interactions of a complex-weighted graph~\cite{Bottcher2024}.

The level of synchrony between oscillators is measured through functions of the state variables, i.e., order parameters or observables (e.g., Kuramoto order parameter). For this reason, the time evolution of these observables is central to describing such collective phenomenon rather than the individual state of each oscillator. This approach is naturally aligned with Koopman theory, developed in the 1930s to formulate classical dynamics using linear operators on spaces of
observables---mirroring the structure of quantum mechanics~\cite{Koopman1931,Koopman1932, VonNeumann1932, Carleman1932, Budisic2012, Joseph2020, Brunton2022}. Under this perspective, a linear operator of interest is the Koopman generator
\begin{align}\label{eq:kooku}
    \mathcal{K} = \bm{p}(\bm{z})^\top \bm{L}_{-1} 
 - \overline{\bm{p}(\bm{z})}^\top \bm{L}_1\,, %\sum_{j,k \in\mathcal{V}}\left(A_{jk}z_k - \bar{A}_{jk}\bar{z}_kz_j^2\right)\partial_j\,,
 % \partial_t + 
\end{align}
where $\bm{p} = (p_1\;\cdots\;p_N)$ with $p_j$ as in Eq.~\eqref{eq:kuramotoz} and 
\begin{align*}
    \bm{L}_{n} = (z_1^{n-1}\partial_1\;\cdots\;z_N^{n-1}\partial_N)^\top\,,
\end{align*} %$\partial_t$ and 
while $\partial_j$ is the partial derivative with respect to $z_j$. The Koopman generator is a total time derivative and generates the time-evolution operator, that is, the Koopman operator~\cite{Budisic2012, Joseph2020, Brunton2022}. 

The linearity of such generator allows us to leverage spectral theory, i.e., find Koopman eigenvalues $\lambda$ and eigenfunctions $\psi$ satisfying $\mathcal{K}\psi = \lambda\psi$~\footnote{The spectral theory of such operators is a rich subject that draws on a range of tools from functional analysis. In this paper, however, we work only at a formal level: we do not address the regularity of the eigenfunctions, the convergence of the expansions, or the functional spaces on which the operator acts and in which the eigenfunctions are defined.}. In Ref.~\cite{Thibeault2026}, we have shown how to find network motifs admitting constants of motion, through (1) eigenfunctions of $\mathcal{K}$, (2)~the Lie-algebraic structure of $\mathcal{K}$, and (3) the Lie continuous symmetries, established through a commutation relation with~$\mathcal{K}$. Indeed, for (1), recall that an eigenfunction $C(\bm{z})$ of $\mathcal{K}$ with eigenvalue zero is a constant of motion of the dynamics. 

Moreover, if one finds two eigenfunctions $\psi_1, \psi_2$ of $\mathcal{K}$ with respective eigenvalues $\lambda_1, \lambda_2$, then $\psi_1^a\psi_2^b$ with $a,b\in\mathbb{C}$ is also an eigenfunction with eigenvalue $a\lambda_1 + b\lambda_2$. Thus, one can choose the vector $(a\,\,b)$ to be orthogonal to $(\lambda_1\,\,\lambda_2)$ to make $\psi_1^a\psi_2^b$ a constant of motion. We will use this classical approach~\cite{Darboux1878, Goriely2001, Zhang2017} in Sec.~\ref{subsec:partial_integration_monomial} to get monomial conserved quantities. Before that, we construct the partially integrable Kuramoto model admitting monomial eigenfunctions and conserved cross-ratios

\section{Construction of the partially integrable model}
In Ref.~\cite{Thibeault2026}, we have obtained the necessary and sufficient conditions to have monomial eigenfunctions of $\mathcal{K}$
\begin{align}
    \psi(\bm{z}) = z_1^{\mu_1}...z_N^{\mu_N} =: z^{\bm{\mu}}\,,\quad \bm{\mu}\in\mathbb{R}^N\,,
\end{align}
with eigenvalue $i\bm{\mu}^\top\bm{\omega}$ and conserved cross-ratios
\begin{equation}\label{eq:cross_ratio}
c_{abcd}(\bm{z}) = \frac{(z_{c}-z_{a})(z_{d}-z_{b})}{(z_{c}-z_{b})(z_{d}-z_{a})}%\,\, a,b,c,d \in \mathcal{V}
\end{equation}
for non-identical indices $a,b,c,d\in\mathcal{V}$ (an eigenfunction of $\mathcal{K}$ with null eigenvalue) for the Kuramoto model described in Sec.~\ref{sec:different_descriptions}. 

Using these conditions, which we will recall and adapt in the following subsections, we aim to construct a modular random graph of Kuramoto oscillators where each instance admits a fixed number of monomial eigenfunctions and conserved cross-ratios. 

More explicitly, we introduce the partition $P = \{\mathcal{M}_1,...,\mathcal{M}_m, \mathcal{C}_1,...,\mathcal{C}_c, \mathcal{P}\}$ of the vertex set $\mathcal{V}$ related to the graph with complex weight matrix $A$ [Eq.~\eqref{eq:complex_matrix}]. In the partition, a subset $\mathcal{M}_\tau$ of size $d_\tau$ admits one monomial eigenfunction, a subset $\mathcal{C}_\gamma$ of size $n_\gamma \geq 4$ admits $n_\gamma - 3$ conserved cross-ratios and the subset $\mathcal{P}$ of size $p$ does not admit any considered conserved quantities or eigenfunctions. The dimensions associated to the partition are summarized in Table~\ref{tab:random_graph}. 

The parts $\mathcal{M}_1,...,\mathcal{M}_m$ satisfy the conditions from the first theorem of Ref.~\cite{Thibeault2026} while the parts $\mathcal{C}_1,...,\mathcal{C}_c$ satisfy the third theorem of Ref.~\cite{Thibeault2026}. For the sake of completeness, let us recall the conditions and apply them to the respective parts.

\subsection{Monomial eigenfunctions}

A monomial eigenfunction $\psi(\bm{z}) = z^{\bm{\mu}}$ corresponds to an observable that evolves according to the linear equation $\dot{\psi} = i\tilde{\omega}\,\psi$, where $\tilde{\omega} = \bm{\mu}^\top \bm\omega$ defines a new frequency. Thus, the time evolution of the monomial eigenfunction is $\psi(\bm{z}(t)) = \exp(i\bm{\mu}^\top \bm\theta(0))\exp(i\tilde{\omega} \,t)$. Intuitively, this means that the observable just rotates independently at a frequency $\tilde{\omega}$ on the unit circle starting from $\exp(i\bm{\mu}^\top \bm\theta(0))$. The real form of this observable also gives another concrete perspective. Indeed, $\phi(t) := \Arg(\psi(\bm{z}(t))) = \bm{\mu}^\top\tilde{\bm{\theta}}(t)$, where $\tilde{\bm{\theta}}(t) = \bm\omega t + \bm{\theta}(0)$ is the solution of the non-interacting system (Eq.~\eqref{eq:kuramoto} with $W = 0$). Therefore, when there is a monomial eigenfunction, there is a reference frame rotating at frequency $\tilde{\omega}$ that freezes $\phi(t)$ in time.

In fact, having a monomial eigenfunction for each part implies the existence of constants of motion. We do not need to dive into this subject for the construction of the partially integrable model, so we provide the details later in Sec.~\ref{subsec:partial_integration_monomial}. Let us focus on the conditions to have such monomial eigenfunctions.

Let the part $\mathcal{M}_\tau \subset \mathcal{V}$ be such that $|\alpha_{jk}| < \pi/2$ for all $j,k\in\mathcal{M}_\tau$ and for all $\tau \in \{1,...,m\}$. Let $\bm \mu_\tau = (\mu_{\tau 1}\,\,\,\cdots\,\,\,\mu_{\tau N})^\top \in \mathbb{R}^N$ satisfy $\mu_{\tau j} \neq 0$ if and only if $j \in \mathcal{M}_\tau$.  The first theorem, reformulated from Ref.~\cite{Thibeault2026}, states that there exists a $\bm{\mu}_\tau$ such that $z^{\bm \mu_\tau}$ is an eigenfunction of $\mathcal{K}$ in Eq.~\eqref{eq:kooku} if and only if: 
\begin{enumerate}[label=\normalfont1.\arabic*.,ref=1.\arabic*]
    \item \label{itm:1.1} $W_{jk} = 0$ for all $j\in \mathcal{M}_\tau$ and $k\in \mathcal{V}\setminus\mathcal{M}_\tau$;
    \item \label{itm:1.2} $W_{jk}\neq 0$ whenever $W_{kj}\neq 0$ for all $j,k\in \mathcal{M}_\tau$ ;
    \item \label{itm:1.3} $W_{i_1 i_2} ... W_{i_{\eta-1}i_{\eta}}W_{{i_{\eta} i_1}} = W_{{i_1 i_{\eta}}}W_{i_{\eta}i_{\eta-1}} ... W_{i_2 i_1}$\\ for all sequences $i_1, i_2, ..., i_\eta$ of elements of $\mathcal{M}_\tau$;
    \item \label{itm:1.4} 
    $\alpha_{jk} = -\alpha_{kj}$ whenever  $j,k\in \mathcal{M}_\tau$,  $j\neq k$, $W_{jk}\neq 0$.
\end{enumerate}
If $z^{\bm{\mu}_\tau}$ is an eigenfunction, then its eigenvalue is $i\bm{\mu}_\tau^\top\bm{\omega}$.

The first condition implies that the subgraph with vertex set $\mathcal{M}_\tau$ is a source within the whole graph. In matrix terms, the second and third conditions are equivalent to the condition that the submatrix related to $\mathcal{M}_\tau$ is symmetrizable, i.e., the submatrix can be made symmetric by multiplying each of its rows by a respective constant. The fourth condition ensures that the phase-lag submatrix related to $\mathcal{M}_\tau$ is skew-symmetric. 

When a weight matrix $W$ of a connected graph is given, the exponents of the monomials are easily constructed. By starting from some $\ell\in\mathcal{M}_\tau$, choose an initial nonzero $\mu_{\tau\ell}$ and walk through the subgraph induced by $\mathcal{M}_\tau$, setting iteratively \mbox{$\mu_j = (W_{k j}/W_{jk}) \mu_k$}. Yet, in this paper, we shall rather construct $W$ given the exponents and a symmetric matrix, thus making symmetrizable blocks in $W$ (Table~\ref{tab:random_graph} and Fig.~\ref{fig:matrix_example}(b)).

\subsection{Conserved cross-ratios}

A cross-ratio $c_{abcd}$ with $a,b,c,d \in \mathcal{C}_\gamma$ is conserved if and only if the vertices $a$, $b$, $c$, $d$ of the graph described by the complex matrix in Eq.~\eqref{eq:complex_matrix} have the same:
\begin{enumerate}[label=\normalfont2.\arabic*.,ref=2.\arabic*, itemsep=0pt, topsep=2pt, parsep=0pt, partopsep=0pt]
\item \label{itm:2.1} outgoing interactions within $\{a,b,c,d\}$, i.e.,
\begin{align*}
\begin{aligned}
    A_{ba} &= A_{ca} = A_{da} =: \mathcal{A}_{\gamma a}\,,\\ 
    A_{ab} &= A_{cb} = A_{db} =: \mathcal{A}_{\gamma b}\,, 
\end{aligned}
\qquad
\begin{aligned}
    A_{ac} &= A_{bc} = A_{dc} =: \mathcal{A}_{\gamma c}\,, \\
    A_{ad} &= A_{bd} = A_{cd} =: \mathcal{A}_{\gamma d}\,,
\end{aligned}
\end{align*}

\item  \label{itm:2.2} incoming interactions from $\mathcal{V}\setminus\{a, b, c, d\}$, i.e.,
\begin{equation*}
    A_{ak} = A_{bk} = A_{ck} = A_{dk}\,,\quad \forall k\in\mathcal{V}\setminus\{a, b, c, d\}\,,
\end{equation*}

\item \label{itm:2.3} shifted natural frequencies
\begin{align*}
\omega_j - 2\,\mathrm{Im}(\mathcal{A}_{\gamma j}) = \omega_k - 2\,\mathrm{Im}(\mathcal{A}_{\gamma k}),\quad \forall j,k \in \{a,b,c,d\}\,.
\end{align*}

\end{enumerate}

These three conditions are satisfied for $\mathcal{C}_\gamma$, for all $\gamma \in \{1,...,c\}$, leading to $n_\gamma - 3$ conserved cross-ratios for each part. 

Considering the partition $P$ along with the conditions~\ref{itm:1.1}-\ref{itm:1.4} and the conditions~\ref{itm:2.1}-\ref{itm:2.3} leads to the random matrix model described in Table~\ref{tab:random_graph}. The probability density functions for the random variables are arbitrary. For the matrix $C$, we have only added a Bernoulli matrix $Q$ to have the freedom to tune the graph density. 

Numerically, we choose to make the model a specific type of weighted stochastic block model, such that we can define the probabilities of connections between the parts $\mathcal{M}_1,...,\mathcal{M}_m, \mathcal{C}_1, ..., \mathcal{C}_c, \mathcal{P}$~\cite{Thibeault2026_koopman_kuramoto}.

\begin{table*}[t]
\caption{\\ 
Random matrix admitting $m$ monomial eigenfunctions and $\sum_{\gamma=1}^c(n_{\gamma} -3)$ conserved cross-ratios in the Kuramoto model~\eqref{eq:kuramoto}. An illustration of the matrix $W$ is given in Fig.~\ref{fig:matrix_example} and the code to generate instances of the random matrix is in Ref.~\cite{Thibeault2026_koopman_kuramoto}.}
\begin{ruledtabular}
\label{tab:random_graph}
\begin{tabular}{ c c }

\textbf{Model} & \vspace{-0.7cm}\\

\multicolumn{2}{c}{\parbox[t]{16cm}{\begin{align}\label{eq:complex_random_matrix} 
A = \frac{1}{2}\left( W \circ e^{-i\alpha} + i\diag(\bm{\omega})\right)\,,\quad \text{where}\,\quad W = \begin{pmatrix}
        D^{-1}S\\
        MC^\top\\ 
        *
    \end{pmatrix}\,,\quad \alpha = \begin{pmatrix}
        \kappa\\
        M\chi^\top\\ 
        *
    \end{pmatrix}\,,\quad \bm{\omega} = (\omega_j)_{j\in\mathcal{V}}
\end{align}}}

\vspace{-0.6cm}\\ 

\\
\textbf{Dimension} &
\\

\parbox[t]{1cm}{$m$} & \parbox[t]{15cm}{\justifying \noindent Number of motifs admitting a monomial eigenfunction}\\
\parbox[t]{1cm}{$c$} & \parbox[t]{15cm}{\justifying \noindent Number of motifs admitting conserved cross-ratios}\\
\parbox[t]{1cm}{$p$} & \parbox[t]{15cm}{\justifying \noindent Number of vertices in the non-integrable part $\mathcal{P}$}\\
\parbox[t]{1cm}{$d_\tau$} & \parbox[t]{15cm}{\justifying \noindent Number of vertices in the $\tau$-th motif $\mathcal{M}_{\tau}$ admitting a monomial eigenfunction with $\tau\in \{1,...,m\}$} \\
\parbox[t]{1cm}{$n_{\gamma}\geq 4$} & \parbox[t]{15cm}{\justifying \noindent Number of vertices in the $\gamma$-th motif $\mathcal{C}_\gamma$ admitting conserved cross-ratios with $\gamma \in \{1,...,c\}$}\\
\parbox[t]{1cm}{$N$} & \parbox[t]{15cm}{\justifying \noindent Total number of vertices in $\mathcal{V} = \bigcup_{\tau}\mathcal{M}_{\tau}\cup\bigcup_{\gamma}\mathcal{C}_{\gamma}\cup\mathcal{P}$, equal to $\sum_\tau d_\tau + \sum_\gamma n_\gamma + p$}\\

\vspace{-0.1cm}
\\
\noindent\textbf{Matrix} &
\\
 \parbox[t]{1cm}{$D$} & \parbox[t]{15cm}{\justifying \noindent$\sum_\tau d_\tau \times \sum_{\tau}d_\tau$ invertible diagonal random matrix $\diag(\mu_{1,1}, ..., \mu_{1,d_1}, \mu_{2,d_1+1}, ..., \mu_{2, d_1 + d_2}, ..., \mu_{m, \sum_\tau d_\tau})$}\\ % \textcolor{red}{$\diag(\mu_{11}, ..., \mu_{1d_1}, ..., \mu_{m1}, ..., \mu_{md_m})$} with real random variables $\mu_{\tau, j}\neq 0$}
     \parbox[t]{1cm}{$S$} & \parbox[t]{15cm}{\justifying \noindent$\sum_{\tau}d_\tau\times N$ matrix $(B+B^\top \quad 0)$ formed by the concatenation of (1) a $\sum_{\tau}d_\tau\times \sum_{\tau}d_\tau$ block diagonal symmetric matrix $B+B^\top$, where $B$ is a real random block diagonal matrix with blocks of respective sizes $d_1,...,d_m$ and (2) a $\sum_{\tau}d_\tau\times (\sum_\gamma n_\gamma + p)$ matrix block of zeros}\\ %  to satisfy condition~1.1
     \parbox[t]{1cm}{$\kappa$} & \parbox[t]{15cm}{\justifying \noindent $\sum_{\mu}d_\mu \times N$ matrix $(\beta - \beta^\top\quad 0)$ formed by the concatenation of (1) a $\sum_{\tau}d_\tau\times \sum_{\tau}d_\tau$ block diagonal skew-symmetric matrix $\beta-\beta^\top$, where $\beta$ is a real random block diagonal matrix with blocks of respective sizes $d_1,...,d_m$ and elements satisfying $|\beta_{jk} - \beta_{kj}| < \pi/2$ and (2) a $\sum_{\tau}d_\tau\times (\sum_\gamma n_\gamma + p)$ matrix block of zeros}\\
     \parbox[t]{1cm}{$M$} & \parbox[t]{15cm}{\justifying \noindent$\sum_\gamma n_\gamma \times c$  membership matrix (binary, nonrandom, column-orthogonal) for the groups $\mathcal{C}_1,...,\mathcal{C}_c$, with single-membership assignment (each vertex/oscillator belongs to exactly one group)}\\
     \parbox[t]{1cm}{$C$} & \parbox[t]{15cm}{\justifying \noindent$N \times c$ real random matrix, e.g., $C = Q\circ R$ where $Q$ is a $N \times c$ Bernoulli matrix and $R$ is a $N \times c$ real random weight matrix} \\
     % Regarding the matrix $C$, one possible choice is $C = U\circ V$ where $U$ is a $N \times c$ Bernoulli random matrix (possibly defined by blocks) and $V$ is a $N \times c$ real random matrix with the weights defined from any desired distribution. 
    \parbox[t]{1cm}{$\chi$} & \parbox[t]{15cm}{\justifying \noindent$N \times c$ real random matrix, where $|\chi_{j\gamma}|<\pi/2$ for all $j\in\mathcal{V}$, $\gamma\in\{1,...,c\}$}\\
     \parbox[t]{1cm}{$\mathcal{A}$} & \parbox[t]{15cm}{\justifying \noindent $c \times N$ complex random matrix equal to $\frac{1}{2}C^\top\circ e^{-i\chi^\top}$}\\
     \parbox[t]{1cm}{$*$} & \parbox[t]{15cm}{\justifying \noindent Arbitrary block of size $p \times N$ related to the non-integrable part $\mathcal{P}$}
\vspace{0.2cm}\\

\noindent\textbf{Frequency} &
\\ %Equals $w_j +  W_{jj}\sin(\alpha_{jj})$, where the second term is only meant to cancel the diagonal contribution of $W$ and $\alpha$ in $A$. 
 \parbox[t]{1cm}{$\omega_j$} & \parbox[t]{15cm}{\justifying \noindent 
 If $j \in \mathcal{M}_\tau$ for $\tau\in\{1,...,m\}$ or $j\in \mathcal{P}$, $\omega_j$ is chosen at random with no restriction. For all $j\in \mathcal{C}_\gamma$ and $\gamma \in\{1,...,c\}$, $\omega_j = \omega_{\ell_\gamma} + 2\,\Imag(\mathcal{A}_{\gamma j} - \mathcal{A}_{\gamma \ell_\gamma})$ where $\ell_{\gamma}$ is any index within $\mathcal{C}_\gamma$ and $\omega_{\ell_\gamma}$ is random}
\end{tabular}
\end{ruledtabular}
\end{table*}
%\end{widetext}

\subsection{Partially integrable model}
\label{subsec:partially_integrable_model}
The partition and the random model for $A$ of Table~\ref{tab:random_graph} imply a separation of the dynamics such that the equations for the oscillators in the non-integrable part $\mathcal{P}$ remain unchanged, the equations admitting monomial eigenfunctions are
\begin{align}\label{eq:kumonomez}
    \dot{z}_j = i\omega_jz_j + \frac{1}{2\mu_{\tau j}} \sum_{k\in\mathcal{M}_{\tau}}S_{jk} (z_ke^{-i\kappa_{jk}} - \bar{z}_kz_j^2e^{i\kappa_{jk}})\,,
\end{align}
for all $j\in\mathcal{M}_\tau$ and $\tau\in\{1,...,m\}$, and the equations admitting conserved cross-ratios are
\begin{align}\label{eq:kucross}
    \dot{z}_j = i\Omega_{\gamma}z_j + \sum_{k=1}^N (\mathcal{A}_{\gamma k}z_k - \bar{\mathcal{A}}_{\gamma k}\bar{z}_kz_j^2)\,,
\end{align}
for all $j\in\mathcal{C}_\gamma$ and $\gamma\in\{1,...,c\}$, where condition~\ref{itm:2.3} has induced a new frequency $\Omega_{\gamma} =  \omega_{\ell_{\gamma}} - 2\,\mathrm{Im}(\mathcal{A}_{\gamma \ell_{\gamma}})$. The role of the natural frequency shifts in condition~\ref{itm:2.3} is now evident: they ensure that the oscillators in $\mathcal{C}_{\gamma}$ share a common effective natural frequency $\Omega_{\gamma}$ that depends on the original natural frequencies, the phase lags, and the weight matrix. This comes as no surprise: conditions~\ref{itm:2.1}-\ref{itm:2.3} imply that the dynamical equations in $\mathcal{C}_{\gamma}$ are identical, although the contribution of each oscillator $\mathcal{C}_{\gamma}$ can vary through the heterogeneity of their out degrees and their initial conditions~\cite{Thibeault2026}.

The Koopman generator for the oscillators in $\mathcal{C}_\gamma$ is 
\begin{align}\label{eq:kookucross}
     \mathcal{K}_{\gamma} &=  \rho_\gamma(\bm{z}) L_{-1}^{\gamma} + i \Omega_{\gamma} L_0^{\gamma} - \overline{\rho_\gamma(\bm{z})} L_1^{\gamma}\,,
\end{align}
where $L_n^{\gamma} = \sum_{j\in \mathcal{C}_\gamma} z_j^{n+1} \partial_j$ and
\begin{align}\label{eq:rho}
    \rho_{\gamma}(\bm{z}) = \sum_{k=1}^N \mathcal{A}_{\gamma k}z_k\,.
\end{align}
Through Eq.~\eqref{eq:kookucross}, it is clear that the cross-ratios related to each partially integrable part $\mathcal{C}_\gamma$ are constants of motion, since they are the joint invariants of $L_{-1}^\gamma, L_0^\gamma, L_1^\gamma$ for each~$\gamma$~\cite{Thibeault2026}. 

The partition thus ultimately leads to the separation of the Koopman generator as
\begin{align}\label{eq:splitted_generator}
    \mathcal{K} = \sum_{\tau=1}^m\mathcal{J}_{\tau} + \sum_{\gamma=1}^c\mathcal{K}_{\gamma} + \mathcal{K}_{\mathrm{NI}}\,,
\end{align}
where $\mathcal{J}_\tau$ is the Koopman generator of Eq.~\eqref{eq:kumonomez} admitting monomial eigenfunctions, $\mathcal{K}_{\gamma}$ is the Koopman generator of Eq.~\eqref{eq:kucross} admitting conserved cross-ratios, while $\mathcal{K}_{\mathrm{NI}}$ is the Koopman generator of the non-integrable part $\mathcal{P}$. This completes the definition of the partially integrable model and we now perform its partial integration.

\begin{figure*}
    \centering
    \includegraphics[width=\linewidth]{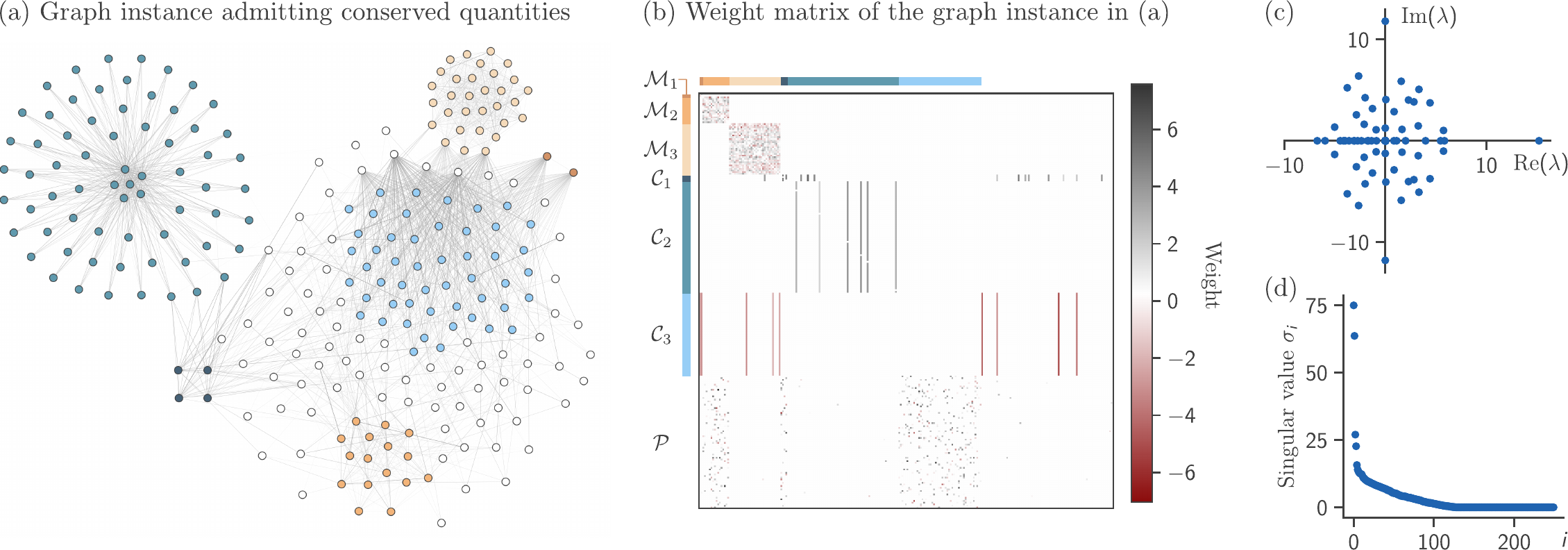}
    
    \vspace{-0.3cm}
    
    \caption{Illustration of (a) a modular, directed, weighted, and signed graph of $N=250$ vertices admitting 2 time-independent monomial constants of motion and 112 conserved cross-ratios. Its weight matrix $W$ in (b) is an instance of the random matrix model of Table~\ref{tab:random_graph} ($A$, $\alpha$, and $\bm{\omega}$ not shown in the figure). (c) The eigenvalues and (d) the singular values of the matrix (b).}
    \label{fig:matrix_example}
\end{figure*}

\section{Partial integration}
\label{sec:partial_integration}
Having now constructed a family of partially integrable Kuramoto models, let us proceed with their partial integration. We first begin by partially integrating the parts admitting a monomial eigenfunction, then we provide an operator-theoretic derivation of WS transformation to finally perform the partial integration of the parts with conserved cross-ratios. 

\subsection{Partial integration of monomial parts}
\label{subsec:partial_integration_monomial}
The existence of a monomial eigenfunction $z^{\bm \mu_{\tau}}$ implies that $z^{\bm{\mu}_{\tau}}e^{-i(\bm{\mu}_{\tau}^\top \bm{\omega})t}$ is conserved, where $i\bm{\mu}_{\tau}^\top \bm{\omega}$ is the related eigenvalue. If $\bm{\mu}_{\tau}^\top\bm\omega = 0$, the time dependence disappears and it is convenient to make a change of variables for the partial integration in such a way that no explicit time dependence is introduced in the vector field, that is, the dynamics remains autonomous.
But this is a specific case, there is no reason for the orthogonality condition $\bm{\mu}_\tau^\top\bm\omega = 0$ to hold in general. 

Yet, it is always possible to make products of different monomial eigenfunctions to obtain monomial constants of motion with no time dependence (as explained in Sec.~\ref{sec:different_descriptions})~\footnote{Another possibility to obtain $m$ constants of motion with no time dependence is that the oscillators in $\mathcal{M}_1,...,\mathcal{M}_m$ also satisfy the conditions of Theorem~2 in Ref.~\cite{Thibeault2026}. In this case, there is a Vandermonde-ratio eigenfunction for each part and one can form a different type of constant of motion, which we do not discuss further in this paper.}. Indeed, suppose that $\mathcal{K}$ admits $m \in \{1,...,N\}$ functionally independent monomial eigenfunctions $z^{\bm\mu_1}, ...,z^{\bm\mu_m}$, where $\bm{\mu}_{\tau}\in\mathbb{R}^N$ for each $\tau\in \{1,...,m\}$, whose corresponding eigenvalues are $\bm \lambda = (i\bm\mu_1^\top \bm{\omega}\,\,\, \cdots\,\,\, i\bm\mu_m^\top \bm{\omega})^\top$. If $\bm\lambda \neq \bm 0$, then the imaginary part of $\bm\lambda$ lies in $\mathbb{R}^m$
and has a $(m - 1)$-dimensional orthogonal complement $\mathscr{A}^\perp\subset \mathbb{R}^m$. One can thus choose $m-1$
linearly independent vectors $\bm a_1,\ldots,\bm a_{m-1}\in \mathscr{A}^\perp$ and get new exponents $\bm{\nu}_{1}, ..., \bm{\nu}_{m-1}$ defined by
\begin{align}\label{eq:nu}
    (\bm{\nu}_{1} \; \cdots \; \bm{\nu}_{m-1}) = (\bm{\mu}_{1} \; \cdots \; \bm{\mu}_{m}) (\bm{a}_{1} \; \cdots \; \bm{a}_{m-1})\,,
\end{align}
such that the corresponding monomials $z^{\bm\nu_1}$, ..., $z^{\bm\nu_{m-1}}$ are eigenfunctions with zero eigenvalues. In other words, they are $m-1$ functionally independent monomial constants of motion.

On the one hand, from the construction summarized in Table~\ref{tab:random_graph}, $(\bm{\mu}_{1} \; \cdots \; \bm{\mu}_{m})$ is a block diagonal matrix with diagonal blocks $(\mu_{1,1} \; \cdots \;\mu_{1,d_1})^\top$, $(\mu_{2,d_1+1}\; \cdots \;\mu_{2, d_1 + d_2})^\top$, and so on. On the other hand, there are several ways to choose $\bm{a}_{1},...,\bm{a}_{m-1}$ in Eq.~\eqref{eq:nu}. One of them is to define $\tilde{\omega}_\tau := \bm{\mu}_\tau^\top \bm{\omega}$ for all $\tau\in\{1,...,m\}$ and set
\begin{align}\label{eq:a}
    \bm{a}_\eta = \tilde{\omega}_{\eta + 1}\bm{e}_\eta - \tilde{\omega}_{\eta}\bm{e}_{\eta + 1}\,,\quad \forall \eta\in\{1,...,m-1\}\,,
\end{align}
where $\bm{e}_\iota$ is a $m$-dimensional unit vector with a 1 at the $\iota$-th component. Equations~(\ref{eq:nu}-\ref{eq:a}) thus imply that
\begin{align}\label{eq:nueta}
    \bm{\nu}_\eta = \tilde{\omega}_{\eta + 1}\bm{\mu}_\eta - \tilde{\omega}_{\eta}\bm{\mu}_{\eta + 1}\,,\quad \forall \eta\in\{1,...,m-1\}\,,
\end{align}
are the exponents of the conserved monomials $z^{\bm{\nu}_1}$,...,$z^{\bm{\nu}_{m-1}}$.

The choice of $\bm{a}_{1},...,\bm{a}_{m-1}$ directly affects the monomial exponents $\bm{\nu}_{1},...,\bm{\nu}_{m-1}$ and, consequently, the change of variables for the partial integration. To define such a change of variables, it is convenient to work with the real equations in terms of the phases and to set $m-1$ of the new variables to be functions of the conserved monomials. Indeed, since a function of a constant of motion is again a constant of motion, $-i\ln(z^{\bm{\nu}_\eta})$ is also conserved for all $\eta$. Recalling that $z_j = e^{i\theta_j}$, we have $-i\ln(z^{\bm{\nu}_\eta}) = \bm{\nu}_\eta^\top\bm\theta$, that is, a conserved linear observable in terms of the phases.

Consider the new variables $\phi_{j_\eta} := \bm{\nu}_\eta^\top\bm\theta$ for all $\eta\in\{1,...,m-1\}$ and for some index $j_\eta\in\mathcal{M}_{\eta}\cup \mathcal{M}_{\eta+1}$, say $j_\eta = \max\,\mathcal{M}_{\eta}$ (i.e., $j_\eta$ is the highest index in $\mathcal{M}_{\eta}$). From there, set $\mathcal{M} = \mathcal{M}_1\cup...\cup\mathcal{M}_m$, $\mathcal{I} = \{j_1,...,j_{m-1}\}$ and define the coordinate vector $\hat{\bm{\phi}} = ((\phi_j)_{j\in \mathcal{M}\setminus \mathcal{I}}, (\phi_j)_{j\in \mathcal{I}})^\top$, where $\phi_j = \theta_j$ for $j\in \mathcal{M}\setminus \mathcal{I}$ and the hat $\,\hat{}\,$ is meant to denote the reordering of the elements of $\bm{\phi} = (\phi_1,...,\phi_{\sum_\tau d_\tau})^\top$. This reordering is meant to separate the time-evolving variables $(\phi_j)_{j\in \mathcal{M}\setminus \mathcal{I}}$ from the constants of motion $(\phi_j)_{j\in \mathcal{I}}$.

In matrix form, one has a linear change of coordinates
\begin{align}\label{eq:monomial_change_coord}
    \hat{\bm{\phi}} = V\hat{\bm{\theta}}\,,
\end{align}
where $\hat{\bm{\theta}} = ((\theta_j)_{j\in \mathcal{M}\setminus \mathcal{I}}, (\theta_j)_{j\in \mathcal{I}})^\top$.
The $\sum_{\tau=1}^md_\tau\times \sum_{\tau=1}^md_\tau$ matrix $V$ is defined as
\begin{align}\label{eq:Vmatrix}
    V = \begin{pmatrix} 
        \hat{I} \\
        \hat{V}
    \end{pmatrix}\,,
\end{align}
where $\hat{I} = (I\quad 0)$ is the concatenation of a $(\sum_{\tau=1}^md_\tau - m + 1)\times (\sum_{\tau=1}^md_\tau - m + 1)$ identity matrix $I$ and a zero block matrix of size $(\sum_{\tau=1}^md_\tau - m + 1)\times (m- 1)$, while $\hat{V} = (\hat{\bm{\nu}}_1\cdots\hat{\bm{\nu}}_{m-1})^\top$ contains the vectors $\hat{\bm{\nu}}_\eta = ((\nu_{\eta j})_{j\in \mathcal{M}\setminus \mathcal{I}}, (\nu_{\eta j})_{j\in \mathcal{I}})^\top$ obtained from reordering the first $\sum_{\tau = 1}^m d_\tau$ elements of $\bm{\nu}_\eta = (\nu_{\eta j})_{j\in \mathcal{V}}$ in Eq.~\eqref{eq:nueta}. 

In real form, Eq.~\eqref{eq:kumonomez} becomes
\begin{align}\label{eq:kumonome}
    \dot{\theta}_j = \omega_j + \frac{1}{\mu_{\tau j}} \sum_{k\in\mathcal{M}_{\tau}}S_{jk} \sin(\theta_k - \theta_j - \kappa_{jk})\,,
\end{align}
for all $j\in\mathcal{M}_\tau$ and $\tau\in\{1,...,m\}$. Applying the change of variables in Eq.~\eqref{eq:monomial_change_coord} yields $\dot{\phi}_j = 0$ for $j\in\mathcal{I}$,
\begin{align}
    \dot{\phi}_j = \omega_j + \frac{1}{\mu_{m j}} \sum_{k\in\mathcal{M}_{m}}S_{jk} \sin(\phi_k - \phi_j - \kappa_{jk})\,,
\end{align}
for $j\in\mathcal{M}_m$, and
\begin{align}
    &\dot{\phi}_j = \omega_j + \frac{1}{\mu_{\eta j}} \sum_{k\in\mathcal{M}_{\eta}\setminus \{j_\eta\}}S_{jk} \sin(\phi_k - \phi_j - \kappa_{jk}) \nonumber
    \\&+ \frac{1}{\mu_{\eta j}} S_{jj_\eta } \sin(\textstyle{\sum_{\ell \in \mathcal{M}\setminus\mathcal{I}}}V^{-1}_{j_\eta, \ell}\phi_\ell - \phi_j - \kappa_{jj_\eta} + \kappa_{\eta})\,,
\end{align}
for $j \in \mathcal{M}_{\eta}\setminus\{j_\eta\}$, $\eta\in\{1,...,m-1\}$, where 
\begin{align}\label{eq:kappa_cte_mvt}
    \kappa_{\eta} % = \sum_{\ell \in \mathcal{I}}V^{-1}_{j_\eta, \ell}\phi_\ell(0) 
    = \sum_{\varepsilon = 1}^{m-1}V^{-1}_{j_\eta, j_\varepsilon}\phi_{j_\varepsilon}(0)\,,\quad \phi_{j_\varepsilon}(0) = \bm\nu_\varepsilon^\top \bm{\theta}(0)\,.
\end{align}
%and $V^{-1}_{j_{\eta}, \ell}$ for all $\ell\in\mathcal{M}$ is understood to be the $(\sum_{\tau=1}^md_\tau - m + 1 + \eta)$-th line of $V^{-1}$ (similarly for $V^{-1}_{j_\eta, j_\epsilon}$), following the previously introduced ordering.
Hence, there is a reduction from $\sum_{\tau = 1}^m d_\tau$ equations to $\sum_{\tau = 1}^m d_\tau - m + 1$ equations for the monomial parts.

The monomial constants of motion thus act as phase lags $\kappa_1, ...,\kappa_{m-1}$ in the coordinates $(\phi_j)_{j \in \mathcal{M}_{\eta}\setminus\mathcal{I}}$ for all $\eta\in\{1,...,m-1\}$. Moreover, there is a special type of higher-order interactions that connect the monomial parts originally disconnected from each other. The above procedure is summarized in Fig.~\ref{fig:fig2}(b).

Let us now turn our focus to the partial integration of the conserved cross-ratio parts $\mathcal{C}_1, ..., \mathcal{C}_c$. In contrast to the monomial parts, the coordinate transformation required for partial integration is not linear and more refined mathematical tools are therefore needed to construct it. We thus devote the next subsection to the derivation of this transformation.

\subsection{Operator-theoretic derivation of WS transformation}
\label{subsec:WS_derivation}

Knowing that some cross-ratios are conserved, the simplest way to achieve the integration is to choose the coordinates of three oscillators within each partially integrable part and use the $n_{\gamma} - 3$ independent cross-ratios as the other coordinates, similar to our treatment of the conserved monomials. Yet, since there are generally more conserved quantities in a part $\mathcal{C}_\gamma$ than in a part $\mathcal{M}_\tau$, one can hope to find a coordinate change to mesoscopic observables, i.e., observables depending on all the oscillators within a part $\mathcal{C}_\gamma$. Ideally, one would also like to have coordinates that help characterizing synchronization within the part.

WS theory is known to yield a natural coordinate system to quantify phase synchronization and it retains a clear interpretation in the infinite-size limit $N\to\infty$, owing to its connection with the Ott-Antonsen Ansatz~\cite{Ott2008, Pikovsky2008, Marvel2009, Pikovsky2011}. Also, for $N\to\infty$, traveling wave analysis and time of travel coordinates clarify the origin of WS transformation~\eqref{eq:ws}, as explained in the original papers (Ref.~\cite[p.2392]{Watanabe1993} and Ref.~\cite[Sec. 5.3]{Watanabe1994}). However, the transformation is often used without detailed justification in the literature and explicit derivations of the transformation itself remain scarce. We thus suggest an operator-theoretic derivation, as summarized in Fig.~\ref{fig:fig2}(c).

To begin with, note that the dynamics of any observable $f:\mathbb{T}^N \to \mathbb{C}$ satisfies a linear equation $\dot{f} = \mathcal{K}[f]$. The solution is readily given by $f(\bm{z}(t)) = \exp(t\mathcal{K})f(\bm{z}(0))$, but we would like to have a closed form for $\exp(t\mathcal{K})$ to obtain the desired change of coordinates. For that, we first take advantage of the Koopman generator’s decomposition in Eq.~\eqref{eq:splitted_generator} to define specific observables depending only on the phases of the oscillators in $\mathcal{C}_{\gamma}$. More precisely, consider the observable $g_{\gamma}:\mathbb{T}^N\to \mathbb{C}$ for each $\gamma$ such that $g_{\gamma} = f_{\gamma}\circ r_{\gamma}$ with the projection $r_{\gamma}: (z_j)_{j\in \mathcal{V}} \mapsto (z_j)_{j\in \mathcal{C}_{\gamma}}$ of the state on the subspace spanned by $(z_j)_{j\in \mathcal{C}_{\gamma}}$, and $f_{\gamma}$ sends $(z_j)_{j\in \mathcal{C}_\gamma}$ to a complex number. Then, $\dot{g}_{\gamma} = \mathcal{K}_{\gamma}[g_{\gamma}]$ and $g_{\gamma}(\bm{z}(t)) = \exp(t\mathcal{K}_\gamma)g_{\gamma}(\bm{z}(0))$, but the closed form of $\exp(t\mathcal{K}_{\gamma})$ remains unknown. 

To move forward, we use an insight from WS theory on a class of identical phase dynamics~\cite{Watanabe1993, Watanabe1994}, soon recognized to be related to identical Riccati equations~\cite{Goebel1995}. Strictly speaking, the differential equations for the Kuramoto model are not Riccati equations, but they do have a correspondence. To establish such a correspondence, given a solution curve $\bm{z}$ generated by Eq.~\eqref{eq:kooku} and a function of time $q_{\gamma}$ such that $q_{\gamma}(t) := \rho_{\gamma}(\bm{z}(t))$, define a new time-dependent Koopman generator for each $\gamma$ such that
\begin{align}\label{eq:generator_ric}
    \mathcal{R}_{\gamma}(t) = q_{\gamma}(t) L_{-1}^{\gamma} + i\Omega_\gamma L_0^{\gamma} - \overline{q_{\gamma}(t)}L_1^{\gamma}\,.
\end{align}
Then, $\mathcal{R}_\gamma$ is the generator of a non-autonomous system of identical Riccati equations and generates the same solution as $\mathcal{K}_{\gamma}$ if the initial conditions coincide and $q_{\gamma}(t) = \rho_\gamma(\bm{z}(t))$ (Lemma~\ref{lem:correspondance_nonautonome}).
In other terms, $\mathcal{R}_{\gamma}(t)$ is only aligned with $\mathcal{K}_{\gamma}$ on the solution curve [Fig.~\ref{fig:fig2} (c1)], but this alignment is just enough to pursue partial integration.

\begin{figure*}[t]
    \centering    \includegraphics[width=1\linewidth]{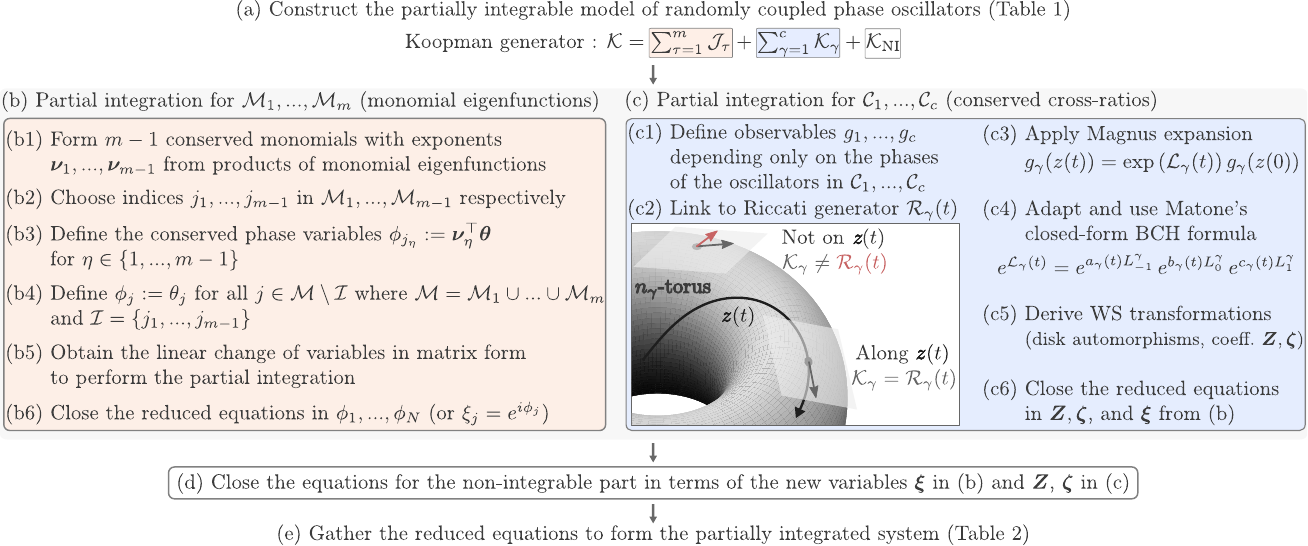}

\vspace{-0.2cm}
    
    \caption{Summary of the procedure to get the partially integrated system from the model admitting monomial eigenfunctions and conserved cross-ratios. (a) We construct the partially integrable model supporting monomial eigenfunctions and conserved cross-ratios, which leads to a separation of the Koopman generator as explained in Sec.~\ref{subsec:partially_integrable_model}. (b) Procedure to partially integrate the parts admitting monomial eigenfunctions. (c) Procedure to partially integrate the parts admitting conserved monomials. In (c2), the link between the Koopman generators $\mathcal{K}_1,...,\mathcal{K}_c$ related to the parts $\mathcal{C}_1,...,\mathcal{C}_c$ and the Koopman generator of Riccati equations $\mathcal{R}_1(t),...,\mathcal{R}_c(t)$ is illustrated. The generator $\mathcal{K}_\gamma$ is only aligned to $\mathcal{R}_\gamma(t)$ along a specific trajectory on the $n_\gamma$-torus. }
    \label{fig:fig2}
\end{figure*}

Under the correspondence $\mathcal{K}_{\gamma} = \mathcal{R}_{\gamma}(t)$, the dynamics of the observable $g_\gamma$ becomes $\dot{g}_\gamma = \mathcal{R}_{\gamma}(t)[g_\gamma]$ for $\gamma \in \{1,...,c\}$, which are non-autonomous, linear, and uncoupled differential equations. Leveraging Magnus expansion~\cite{Magnus1954, *Blanes2009}, we obtain the solution
\begin{align}\label{eq:formal_solution}
    g_{\gamma}(z(t)) = \exp\left(\mathcal{L}_{\gamma}(t)\right)g_{\gamma}(z(0))\,,
\end{align}
where
\begin{equation}
    \mathcal{L}_{\gamma}(t) = B_{\gamma}(t)L_{-1}^{\gamma} + 2iY_{\gamma}(t)L_0^{\gamma} - \overline{B_{\gamma}(t)}L_1^{\gamma}\,,\\
\end{equation}
\begin{equation}
    B_{\gamma}(t) = \sum_{n=1}^\infty b_n^{\gamma}(t)\,,\quad 
    Y_{\gamma}(t) = \sum_{n=1}^\infty y_n^{\gamma}(t)\,,
\end{equation}
while $b_n^{\gamma}(t)$ and $y_n^{\gamma}(t)$ have integral forms and depend on $\Omega_\gamma$ and $q_{\gamma}(t)$ (Sec.~\ref{subsec:partial_integration}). 

The problem in deriving the explicit action of $\exp\left(\mathcal{L}_{\gamma}(t)\right)$ is that $L_{-1}^{\gamma}$, $L_0^{\gamma}$, and $L_1^{\gamma}$ do not commute, i.e., $[L_n^{\gamma}, L_m^{\gamma}] = (n - m)L_{n + m}^{\gamma}$, so the exponential cannot be directly expressed as a product of three exponentials for each operator. Moreover, a closed-form equation for $\exp\left(\mathcal{L}_{\gamma}(t)\right)$ with the BCH formula does not necessarily exist.

Fortunately, in 2015, Matone found the closed-form equation for the exponential of different elements in the Virasoro algebra, including $\mathfrak{sl}_2(\mathbb{C})$ as a subalgebra~\cite{Matone2015}. For our needs, we derived the specific form of Matone's formula for the projective special unitary group $\mathrm{PSU}(1,1)$ (Sec.~\ref{subsec:matone}) :
\begin{align}\label{eq:matone}
    &\exp\left\{\nu(u)\left[VL_{-1} + (U - \bar{U})L_0 -\bar{V}L_1\right]\right\}\\ &= \exp(-(V/\bar{U}) L_{-1})\exp(2\ln \bar{U} L_0)\exp((\bar{V}/\bar{U}) L_1)\,,\nonumber
\end{align}
where $u = (U + \bar{U})/2$, $|U|^2 - |V|^2 = 1$, and
\begin{align*}
    \nu(x) = \frac{1}{2\sqrt{x^2 - 1}}\ln\left(\frac{x + \sqrt{x^2 - 1}}{x - \sqrt{x^2 -1}}\right)\,.
\end{align*}
To use the formula for $\exp\left(\mathcal{L}_{\gamma}(t)\right)$, define
\begin{align*}
    U_{\gamma}(t) := X_{\gamma} + \frac{i Y_{\gamma}(t)}{\nu(X_{\gamma})}\quad\text{and}\quad
    V_{\gamma}(t) := \frac{B_{\gamma}(t)}{\nu(X_{\gamma})}\,,
\end{align*}
where $X_{\gamma}$ is constrained to satisfy $|U_{\gamma}(t)|^2 - |V_{\gamma}(t)|^2 = 1$ for all $\gamma$ and $t$. Despite the apparent complexity of the constraint, we can explicitly solve for $X_{\gamma}$ and obtain the simple relation $X_{\gamma} = \pm \cosh(\sqrt{|B_{\gamma}(t)|^2 - Y_{\gamma}(t)^2})$. One thus gets the desired form
\begin{align*}
    \mathcal{L}_{\gamma} = \nu(X_{\gamma})(V_{\gamma}L_{-1}^{\gamma} + (U_{\gamma} - \overline{U}_{\gamma})L_0^{\gamma} - \overline{V}_{\gamma}L_1^{\gamma})\,.
\end{align*}
Equation~\eqref{eq:matone} allows splitting the exponential and finally provides the action of $\exp\left(\mathcal{L}_{\gamma}(t)\right)$ on the observables:
\begin{align}\label{eq:sol_gen_ric}
    g_{\gamma}(z(t)) = \frac{U_{\gamma}(t)g_{\gamma}(z(0)) + V_{\gamma}(t)}{\overline{V_{\gamma}(t)}g_{\gamma}(z(0)) + \overline{U_{\gamma}(t)}}\,,
\end{align}
that is, a time-dependent disk automorphism in $\mathrm{Aut}(\mathbb{D}) \simeq \mathrm{PSU}(1,1)$. In particular, the choice of observable $g_\gamma(z_1,...,z_N) = z_j$ for $j\in \mathcal{C}_\gamma$ implies that the solutions of the Kuramoto model for $\mathcal{C}_1$, ..., $\mathcal{C}_c$ have the form of disk automorphisms, as expected from WS theory.

There are infinitely many integrals to be solved to get $U_{\gamma}(t)$ and $V_{\gamma}(t)$, but there is an alternative: interpreting them as variables having their own equations of motion. Under this perspective, every solution $(U_{\gamma}(t), V_{\gamma}(t))$ of the yet-to-be-determined differential equations is constrained to start at $(U_{\gamma}(0), V_{\gamma}(0)) = (\pm 1, 0)$ for any initial conditions $(z_j(0))_{j\in \mathcal{C}_\gamma}$ in order for Eq.~\eqref{eq:sol_gen_ric} to be satisfied. Considering that different initial conditions $(\xi_j)_{j\in \mathcal{C}_\gamma}\in \mathbb{T}^{n_{\gamma}}$ of the Kuramoto model lead to different trajectories, it is more informative to have different initial conditions leading to different trajectories for the disk automorphism coefficients. To do so, we express the initial conditions as $\xi_j = (a_{\gamma}w_j + b_{\gamma})/(\overline{b}_{\gamma}w_j + \overline{a}_{\gamma})$, where $|a_{\gamma}|^2 - |b_{\gamma}|^2 = 1$,  $w_j \in \mathbb{T}^1$ for all $j\in \mathcal{C}_\gamma$, and $\gamma\in\{1,...,c\}$. Consequently, 
\begin{align}\label{eq:ws_UV}
    z_j(t) = \frac{u_{\gamma}(t)w_j + v_{\gamma}(t)}{\overline{v_{\gamma}(t)}w_j + \overline{u_{\gamma}(t)}}\,,\quad \forall \,j\in \mathcal{C}_\gamma\,,
\end{align}
where $u_{\gamma}(t) := a_{\gamma}U_{\gamma}(t) + \overline{b}_{\gamma}V_{\gamma}(t)$ and $v_{\gamma}(t) := b_{\gamma}U_{\gamma}(t) + \overline{a}_{\gamma}V_{\gamma}(t)$ satisfy $|u_{\gamma}(t)|^2 - |v_{\gamma}(t)|^2 = 1$. 

Moreover, since $u_{\gamma}(t)$ and $v_{\gamma}(t)$ are not bounded, we make the change of coordinates $u_{\gamma}^2 = \zeta_{\gamma}/(1 - |Z_{\gamma}|^2)$, $v_{\gamma}^2 = \overline{\zeta}_{\gamma}Z_{\gamma}^2/(1 - |Z_{\gamma}|^2)$ where $\zeta_{\gamma} = e^{i\varphi_\gamma}$ and $Z_{\gamma}(t) \in \mathbb{D}$. The change of coordinates applied to Eq.~\eqref{eq:ws_UV} finally leads to 
\begin{align}\label{eq:auto_ric}
    z_j(t) = M_{Z_{\gamma}(t), \zeta_{\gamma}(t)}(w_j) = \frac{\zeta_{\gamma}(t)w_j + Z_{\gamma}(t)}{1 + \zeta_{\gamma}(t)\,\overline{Z_{\gamma}(t)}w_j}\,,
\end{align}
i.e., the complex version of the Watanabe-Strogatz transformation~\eqref{eq:ws}. The transformation that leads to the real, original, form in Eq.~\eqref{eq:ws} is obtained by using trigonometric identities as in Ref.~\cite[IV, A]{Marvel2009}. Our approach thus provides an explicit derivation of the transformation, but also highlights the complex relation between $Z_{\gamma}$, $\zeta_{\gamma}$, and the original coordinates $z_1, ..., z_N$. % Equation~\eqref{eq:auto_ric} is the starting point to partially integrate the parts admitting conserved cross-ratios.

\subsection{Partial integration of cross-ratio parts}
\label{subsec:partial_integration_crossratio}

In the last section, we showed that there are infinitely many integrals to solve to obtain $U_{\gamma}(t)$, $V_{\gamma}(t)$ or equivalently, through a change of variables, $Z_{\gamma}(t)$, $\zeta_{\gamma}(t)$ for all $\gamma \in \{1,...,c\}$. Yet, as mentioned before, we can interpret them as variables having their own equations of motion~: if we find their differential equations in closed form (in terms of $(Z_{\gamma})_{\gamma\in\{1,...,c\}}$, $(\zeta_{\gamma})_{\gamma\in\{1,...,c\}}$, the variables $(\phi_j)_{j\in\mathcal{M}\setminus\mathcal{I}}$ introduced in Sec.~\ref{subsec:partial_integration_monomial}, and  $\bm{z}_{\mathcal{P}} := (z_j)_{j\in \mathcal{P}}$), then the partial integration will be complete.

To perform the partial integration, we adopt the ``algebraic approach'' from Ref.~\cite{Marvel2009} to each part admitting conserved cross-ratios. For the next few steps, we drop the time dependencies to simplify the notation. First, the time derivative of $z_j$ in Eq.~\eqref{eq:auto_ric} for $j \in \mathcal{C}_\gamma$ and for all $\gamma \in\{1,...,c\}$ is
\begin{align}\label{eq:dot_z_wspaper}
    &\dot{z}_j = \nonumber\\&\frac{\dot{Z}_{\gamma} + \dot{Z}_{\gamma}\bar{Z}_{\gamma} - Z_{\gamma}\dot{\bar{Z}}_{\gamma} + i(1-|Z_{\gamma}|^2)\dot{\varphi}_{\gamma}e^{i\varphi_{\gamma}}w_j - \dot{\bar{Z}}_{\gamma}e^{2i\varphi_{\gamma}}w_j^2}{1 + 2 \bar{Z}_{\gamma}e^{i\varphi_{\gamma}}w_j + {\bar{Z}_{\gamma}}^2e^{2i\varphi_{\gamma}}w_j^2\phantom{\sum^{N}}}. 
\end{align}
Second, taking the inverse of Eq.~\eqref{eq:auto_ric}, i.e.,
\begin{align}\label{eq:wj}
    w_j = e^{i\varphi_{\gamma}}\frac{z_j - Z_{\gamma}}{1 - \bar{Z}_{\gamma}z_j}\,, % \qquad\forall t\,\in\mathbb{R},\, j \in \mathcal{C}_\gamma\,,
\end{align}
and substituting into Eq.~\eqref{eq:dot_z_wspaper} yields
\begin{align}
    \dot{z}_j = &\left(\frac{\dot{Z}_{\gamma} - iZ_{\gamma}\,\dot{\varphi}_{\gamma}}{1 - \left|Z_{\gamma}\right|^2} \right) -\left(\frac{\sdot{\bar{Z}}_{\gamma} +i\sbar{Z}_{\gamma}\,\dot{\varphi}_{\gamma}}{1 - \left|Z_{\gamma}\right|^2}\right)z_j^2 \label{eq:} \\&+ i\left(\frac{i(\sbar{Z}_{\gamma}\sdot{Z}_{\gamma}-\sdot{\sbar{Z}}_{\gamma}Z_{\gamma}) + (1 + \left|Z_{\gamma}\right|^2)\dot{\varphi}_{\gamma}}{1 - \left|Z_{\gamma}\right|^2}\right)z_j \nonumber
\end{align}
for all $j\in\{1,...,N\}$. 

Third, for these differential equations to be equivalent to the Riccati equations related to Eq.~\eqref{eq:generator_ric}, the system of differential-algebraic equations must satisfy
\begin{align*}
    q_{\gamma}(t) &= \frac{\sdot{Z}_{\gamma}(t) - iZ_{\gamma}(t)\,\dot{\varphi}_{\gamma}(t)}{1 - \left|Z_{\gamma}(t)\right|^2}\,, \\
    \Omega_{\gamma} &= \frac{i(\sbar{Z}_{\gamma}(t)\sdot{Z}_{\gamma}(t)-\sdot{\sbar{Z}}_{\gamma}(t)Z_{\gamma}(t)) + (1 + \left|Z_{\gamma}(t)\right|^2)\dot{\varphi}_{\gamma}(t)}{1 - \left|Z_{\gamma}(t)\right|^2}
\end{align*}
for all $t$. Solving for $\dot{Z}_{\gamma}$ and $\dot{\varphi}_{\gamma}$ gives
\begin{align*}
    \dot{Z}_{\gamma}(t) &=  q_{\gamma}(t) + i\Omega_{\gamma} Z_{\gamma}(t)- \overline{q_{\gamma}(t)}Z_{\gamma}(t)^2\,,\\
    \dot{\zeta}_{\gamma}(t) &= (i\Omega_{\gamma} + q_{\gamma}(t)\overline{Z_{\gamma}(t)} - \overline{q_{\gamma}(t)}\,Z_{\gamma}(t))\zeta_{\gamma}(t)\,.
\end{align*}

Finally, the relation $q_{\gamma}(t) = \rho_{\gamma}(\bm{z}(t))$ allows going back to the Kuramoto model.
Now, one needs to close $\rho_{\gamma}(\bm{z})$ in terms of $(Z_{\gamma})_{\gamma\in\{1,...,c\}}$, $(\zeta_{\gamma})_{\gamma\in\{1,...,c\}}$, $(\phi_j)_{j\in\mathcal{M}\setminus\mathcal{I}}$, and  $\bm{z}_{\mathcal{P}}$. The partition $P$ splits the sum in Eq.~\eqref{eq:rho} as
\begin{align*}
    \rho_{\gamma}(\bm{z}) = \sum_{k\in\mathcal{M}} \mathcal{A}_{\gamma k}z_k + \sum_{\delta=1}^c \sum_{k\in \mathcal{C}_\delta} \mathcal{A}_{\gamma k}z_k + \sum_{k\in\mathcal{P}} \mathcal{A}_{\gamma k}z_k\,.
\end{align*}
For the first term, by setting $\xi_\ell = e^{i\phi_\ell}$ for $\ell\in\mathcal{M}\setminus \mathcal{I}$, the change of coordinates~\eqref{eq:monomial_change_coord} yields
\begin{align*}
    \sum_{k\in\mathcal{M}} \mathcal{A}_{\gamma k}z_k = \sum_{k\in\mathcal{M}\setminus\mathcal{I}}\mathcal{A}_{\gamma k}\xi_k + \sum_{\varepsilon = 1}^{m-1} \mathcal{A}_{\gamma j_\varepsilon}e^{i\kappa_\varepsilon}\xi^{\bm{v}_\varepsilon}\,,
\end{align*}
where 
\begin{equation}
    \xi^{\bm{v}_\varepsilon} = \prod_{\ell \in \mathcal{M}\setminus \mathcal{I}}\xi_\ell^{V_{j_\varepsilon \ell}^{-1}}
\end{equation}
and $\kappa_\varepsilon$ is given by Eq.~\eqref{eq:kappa_cte_mvt}. Equation~\eqref{eq:auto_ric} closes the second term in terms of $(Z_{\gamma})_{\gamma\in\{1,...,c\}}$, $(\zeta_{\gamma})_{\gamma\in\{1,...,c\}}$ and the third term for the nonintegrable part is already closed.

\begin{table*}[t]
\caption{\\ 
Summary of the partially integrated Kuramoto dynamics on the random graph defined in Table~\ref{tab:random_graph}.}
\begin{ruledtabular}
\label{tab:partially_integrated}
\begin{tabular}{ c }
\parbox[t]{8cm}{\vspace{-1\baselineskip}\begin{align*}
&\textbf{Monomial part}, \eta = \{1,...,m-1\}\\
    &\dot{\xi}_j = i\omega_j\xi_j + \frac{1}{2\mu_{\eta j}} \sum_{k\in\mathcal{M}_{\eta}\setminus \{j_\eta\}}S_{jk} (\xi_ke^{-i\kappa_{jk}} - \bar{\xi}_k\xi_j^2e^{i\kappa_{jk}})\\ & + \frac{1}{2\mu_{\eta j}} S_{jj_\eta } (\xi^{\bm{v}_\eta} e^{i(\kappa_{\eta} - \kappa_{jj_\eta})} - \bar{\xi}^{\bm{v}_\eta}\xi_j^2 e^{-i(\kappa_{\eta} - \kappa_{jj_\eta})} )\,,\,\, j \in \mathcal{M}_{\eta}\setminus\{j_\eta\}\\
    &\dot{\xi}_j = i\omega_j\xi_j + \frac{1}{2\mu_{m j}} \sum_{k\in\mathcal{M}_{m}}S_{jk}(\xi_ke^{-i\kappa_{jk}} - \bar{\xi}_k\xi_j^2e^{i\kappa_{jk}}) \,, \quad j\in\mathcal{M}_m\\
    \\
    &\textbf{Cross-ratio part}, \gamma \in \{1,...,c\}\\
    &\dot{Z}_{\gamma} = F_{\gamma}(\bm{Z}, \bm{\zeta}, \bm{\xi}, \bm{z}_{\mathcal{P}}) + i\Omega_{\gamma} Z_{\gamma} - \overline{F_{\gamma}(\bm{Z}, \bm{\zeta}, \bm{\xi}, \bm{z}_{\mathcal{P}})}Z_{\gamma}^2\,,\\%&&\forall \gamma \in \{1,...,c\}\\
    &\dot{\zeta}_{\gamma} = (i\Omega_{\gamma} + F_{\gamma}(\bm{Z}, \bm{\zeta}, \bm{\xi}, \bm{z}_{\mathcal{P}})\overline{Z}_{\gamma} - \overline{F_{\gamma}(\bm{Z}, \bm{\zeta}, \bm{\xi}, \bm{z}_{\mathcal{P}})}\,Z_{\gamma})\,\zeta_{\gamma}\\
    \\
    &\textbf{Non-integrable part}\\
    &\dot{z}_j = 
    G_{j}(\bm{Z}, \bm{\zeta}, \bm{\xi}, \bm{z}_{\mathcal{P}}) - \overline{G_{j}(\bm{Z}, \bm{\zeta}, \bm{\xi}, \bm{z}_{\mathcal{P}})}z_j^2\,,\qquad j \in  \mathcal{P}
\end{align*}}
\quad
\parbox[t]{8cm}{\vspace{-1\baselineskip}\begin{align*}
&\textbf{Related functions}\\
&\kappa_{\eta}  = \sum_{\varepsilon = 1}^{m-1}V^{-1}_{j_\eta, j_\varepsilon}\phi_{j_\varepsilon}(0)\,,\quad \phi_{j_\varepsilon}(0) = \bm\nu_\varepsilon^\top \bm{\theta}(0)\\
    &\Omega_{\gamma} = \omega_{\ell_{\gamma}} - 2\,\mathrm{Im}(\mathcal{A}_{\gamma \ell_{\gamma}})\\
    &F_{\gamma}(\bm{Z}, \bm{\zeta}, \bm{\xi}, \bm{z}_{\mathcal{P}}) = \sum_{k\in\mathcal{M}\setminus\mathcal{I}}\mathcal{A}_{\gamma k}\xi_k + \sum_{\varepsilon = 1}^{m-1} \mathcal{A}_{\gamma j_\varepsilon} e^{i\kappa_\varepsilon}\xi^{\bm{v}_\varepsilon} \\&\phantom{F_{\gamma}(\bm{Z}, \bm{\zeta}, \bm{\xi}, \bm{z}_{\mathcal{P}}) =}+ \sum_{\delta=1}^c \sum_{k\in \mathcal{C}_\delta} \mathcal{A}_{\gamma k}M_{Z_{\delta}, \zeta_{\delta}}(w_k) + \sum_{k\in \mathcal{P}} \mathcal{A}_{\gamma k}z_k\\
    &G_{j}(\bm{Z}, \bm{\zeta}, \bm{\xi}, \bm{z}_{\mathcal{P}}) = \sum_{k\in\mathcal{M}\setminus\mathcal{I}}A_{j k}\xi_k + \sum_{\varepsilon = 1}^{m-1} A_{j j_\varepsilon}e^{i\kappa_\varepsilon}\xi^{\bm{v}_\varepsilon}\\ &\phantom{G_{j}(\bm{Z}, \bm{\zeta}, \bm{\xi}, \bm{z}_{\mathcal{P}}) = }+ \sum_{\delta =1}^c\sum_{k\in \mathcal{C}_\delta} A_{j k}M_{Z_{\delta}, \zeta_{\delta}}(w_k) + \sum_{k\in \mathcal{P}}A_{j k}z_k 
\end{align*}}

\end{tabular}
\end{ruledtabular}
\end{table*}

Altogether, the closed dynamics for the oscillators in the part $\mathcal{C}_\gamma$ is
 \begin{align*}
    \dot{Z}_{\gamma} &= F_{\gamma}(\bm{Z}, \bm{\zeta}, \bm{\xi}, \bm{z}_{\mathcal{P}}) + i\Omega_{\gamma} Z_{\gamma} - \overline{F_{\gamma}(\bm{Z}, \bm{\zeta}, \bm{\xi}, \bm{z}_{\mathcal{P}})}Z_{\gamma}^2\,, \\%&&\forall \gamma \in \{1,...,c\}\\
    \dot{\zeta}_{\gamma} &= (i\Omega_{\gamma} + F_{\gamma}(\bm{Z}, \bm{\zeta}, \bm{\xi}, \bm{z}_{\mathcal{P}})\overline{Z}_{\gamma} - \overline{F_{\gamma}(\bm{Z}, \bm{\zeta}, \bm{\xi}, \bm{z}_{\mathcal{P}})}\,Z_{\gamma})\,\zeta_{\gamma}\,,% \\
\end{align*}
for all $\gamma\in\{1,...,c\}$, where
\begin{align*}
    F_{\gamma}(\bm{Z}, \bm{\zeta}, \bm{\xi}, \bm{z}_{\mathcal{P}}) &= \sum_{k\in\mathcal{M}\setminus\mathcal{I}}\mathcal{A}_{\gamma k}\xi_k + \sum_{\varepsilon = 1}^{m-1} \mathcal{A}_{\gamma j_\varepsilon}e^{i\kappa_\varepsilon}\xi^{\bm{v}_\varepsilon}\\&+ \sum_{\delta=1}^c \sum_{k\in \mathcal{C}_\delta} \mathcal{A}_{\gamma k}M_{Z_{\delta}, \zeta_{\delta}}(w_k) + \sum_{k\in \mathcal{P}} \mathcal{A}_{\gamma k}z_k\,.
\end{align*}

The closure of the equations related to the oscillators in the non-integrable part $\mathcal{P}$ is similar to the above procedure, which completes the partial integration. The partially integrated system is presented in Table~\ref{tab:partially_integrated}. Two examples (without monomial parts) are provided in SI (sec.~\ref{subsec:basic_cases}). All in all, we started with 
\begin{equation*}
    N = \sum_{\tau=1}^m d_{\tau} + \sum_{\gamma=1}^c n_{\gamma} + p 
\end{equation*}
equations and we obtained a reduced system of dimension
\begin{align*}
    n = \sum_{\tau=1}^{m-1} (d_{\tau} - 1) + d_m + 3c + p\,.
\end{align*} 
Hence, there are
\begin{align*}
    q = N - n = (m-1) + \sum_{\gamma=1}^c (n_{\gamma} - 3)
\end{align*}
functionally independent constants of motion.

\section*{Conclusion}

We introduced a modular random network of Kuramoto oscillators admitting different conserved quantities and demonstrated how to perform the partial integration of the model with operator-theoretic tools. For the latter, we used monomial eigenfunctions of the Koopman generator and we provided a derivation of the Watanabe–Strogatz transformation, based on Magnus expansion and a closed form of the BCH formula introduced by Matone~\cite{Matone2015}.

Although we have applied our framework only to the Kuramoto model, it extends to a broader class of phase dynamics, including the Winfree model~\cite{Winfree1967} and the theta model~\cite{Ermentrout1986}, and we expect it to be useful for other general oscillator dynamics as well~\cite{Lohe2009, Lohe2010, Lohe2018, Tanaka2014, Chandra2019a, Bronski2020, Lipton2021, Gong2019, Cestnik2024}. In particular, considering Riccati-type dynamics with Koopman generator
\begin{align}\label{eq:gen}
     \mathcal{K}_{t,\mathbf{Z}} =  a(t, \mathbf{Z}) \,L_{-1} + b(t, \mathbf{Z})\, L_0 + c(t, \mathbf{Z})\,L_1\,,
\end{align}
where $\mathbf{Z} = (Z_1,...,Z_N)\in\mathbb{C}^N$ and $L_n$ defined with $Z_1,...,Z_N$, one can at least formally apply the procedure suggested in Fig.~\ref{fig:fig2}(c). Indeed, one can still relate the generator to the one of a Riccati equation ($\mathcal{R}_{t,\mathbf{Z}} =  \alpha(t) \,L_{-1} + \beta(t)\, L_0 + \gamma(t)\,L_1$), apply Magnus expansion, and use Matone's formula for $\mathfrak{sl}_2(\mathbb{C})$ instead of $\mathfrak{psu}(1,1)$. 

Yet, further study needs to be done to turn the suggested procedure into a rigorous mathematical derivation by controlling the convergence of the series arising in the Magnus expansion, establishing the formal proof for Matone's formula, and specifying the functional spaces in which $a$, $b$, $c$ in Eq.~\eqref{eq:gen} must belong for the procedure to apply. Since Matone's formula is available for Virasoro algebra~\cite{Matone2015} (see also Ref.~\cite{Matone2016}), it is not excluded that one could find other relevant applications to oscillator dynamics. 

Future work could also investigate the spectrum of the introduced random matrix and its potential connections to conserved quantities and specific synchronization phenomena. The framework may also be generalized to address the partial integration of phase-oscillator networks containing motifs that admit Vandermonde-ratio eigenfunctions and symmetry-generated constants of motion~\cite{Thibeault2026}. In addition, we have not yet carried out a systematic analysis of the persistence of these constants of motion when the underlying random graph is perturbed. More broadly, even when a system is not exactly partially integrable, it may remain approximately partially integrable, allowing for a substantial reduction of the governing ODEs through an appropriate choice of observables.

\section*{Code availability}

The \textit{Python} code for this work is available on Zenodo~\cite{Thibeault2026_koopman_kuramoto}. In particular, the script ``generate\_integrability\_partitioned\_weight\_matrix.py" allows generating realizations of the random graph in Table~\ref{tab:random_graph}.

\section*{Acknowledgments}

We thank Renaud Lambiotte for his constructive comments on the manuscript. This work was supported by the Fonds de recherche du Qu\'ebec -- Nature et technologies (V.T., P.D.) and the Natural Sciences and Engineering Research Council of Canada (A.A., B.C., P.D.).

\bibliography{paper_kuramoto_koopman}

%\putbib[mybib]
%\end{bibunit}

\clearpage
%=~=~=~=~=~=~=~=~=~=~=~=~=~=~=~=~=~=~=~=~=~=~=~=~=~=~=~=~=~=~=~=~=~=~=~=~=~=~=~=
% Supplementary Information
% =~=~=~=~=~=~=~=~=~=~=~=~=~=~=~=~=~=~=~=~=~=~=~=~=~=~=~=~=~=~=~=~=~=~=~=~=~=~=~=

\setcounter{equation}{0}
\setcounter{figure}{0}
\setcounter{section}{0}
\setcounter{page}{1}
\setcounter{theorem}{0}
\setcounter{table}{0}

\renewcommand{\theequation}{S\arabic{equation}}
\renewcommand{\thefigure}{S\arabic{figure}}
\renewcommand{\thetable}{S\Roman{table}}
\renewcommand{\thetheorem}{S\arabic{theorem}}
\renewcommand{\thelemma}{S\arabic{lemma}}
\renewcommand{\theHfigure}{Supplement.\thefigure}  % Imp.: otherwise, might confuse hyperref
\renewcommand{\thesection}{S\Roman{section}}
\thispagestyle{empty}

\onecolumngrid

%\normalsize

\let\addcontentsline\oldaddcontentsline % Restore \addcontentsline

\centerline{\large\textbf{Partially integrable random graph of Kuramoto oscillators}}

\vspace{0.1cm}

\centerline{\textbf{--- Supplementary information ---}}

\tableofcontents

% \begin{align*}
%     \bm{\omega} = (\omega_1,...,\omega_{\sum_\gamma n_\gamma}, w_1, w_1 - 2\Imag(\mathcal{A}_{12} - \mathcal{A}_{\gamma 1}), ..., w_1 - 2\Imag(\mathcal{A}_{1n_1} - \mathcal{A}_{11}) )
% \end{align*} % w_1 - 2\Imag(\mathcal{A}_{12} - \mathcal{A}_{\gamma 1}), ..., w_1 - 2\Imag(\mathcal{A}_{1n_1} - \mathcal{A}_{11})

\section{From the Kuramoto model to Riccati equations}
\label{subsec:descriptions}

In Ref.~\cite{Thibeault2026}, we have shown the following lemma.
\begin{lemma}\label{lem:correspondance_complexe_A}
    With $z_j(t) = e^{i\theta_j(t)}$, the initial value problem of the Kuramoto model is equivalent to 
\begin{align}
    \dot{z}_j(t) &= \sum_{k=1}^NA_{jk}z_k(t) - \left(\sum_{k=1}^N\bar{A}_{jk}\bar{z}_k(t) \right)z_j(t)^2 \label{eq:zkurA}\\
    z_j(0) &= e^{i\vartheta_j}  \in \mathbb{T}\label{eq:z0kurA}\,,
\end{align}
where $A$ is a complex matrix of interactions satisfying
\begin{align}
    A = \frac{1}{2}\left(W \circ e^{-i\alpha} + i\diag(\bm{\omega})\right)\,,
\end{align}
where $e^{-i\alpha} = (e^{-i\alpha_{jk}})_{j,k}$, $\bm{\omega} = (\omega_1,...,\omega_N)$, $\circ$ is the element-wise product and $\diag(W) = \diag(\alpha) = \bm{0}$. There exists a constant $a > 0$ such that the problem~(\ref{eq:zkurA}-\ref{eq:z0kurA}) possesses a unique solution $z_1(t),..., z_N(t)$ on $t\in[-a, a]$.
\end{lemma}
To perform the partial integration of the Kuramoto model, it is useful to establish a correspondence with one solution of the model (thus, for a given initial condition) and one solution of a set of Riccati equations representing a non-autonomous dynamical system.
\begin{lemma}\label{lem:correspondance_nonautonome}
    Let the initial value problem of Riccati be
    \begin{align}
        \dot{u}_j(t) &= q_j(t) - \overline{q_j(t)}u_j(t)^2\,,\qquad j\in\{1,...,N\}\label{eq:riccati}\\
        u_j(0) &= d_j \in \mathbb{T} \label{eq:uCI}
    \end{align} % $q_{-1}(t) = \overline{q_1(t)}$
    such that $u_j(t)\in\mathbb{T}$ for all $t \in [-a, a]$ with $a > 0$ and the arbitrary complex function $q_j$ defines the problem. Moreover, consider some solution of the Kuramoto model [Lemma~\ref{lem:correspondance_complexe_A}] $z_1(t), ..., z_N(t)$ for $t \in [-a, a]$ with $z_1(0) = \exp(i\vartheta_1), ..., z_N(0) = \exp(i\vartheta_N)$ and a complex matrix of interactions $A$. If the Riccati problem in Eqs.~(\ref{eq:riccati}-\ref{eq:uCI}) satisfies $d_j = \exp(i\vartheta_j)$ and is defined with
    \begin{align}\label{eq:pj_to_qj}
        q_j(t) = \textstyle{\sum_{k=1}^N} A_{jk}z_k(t)\,, \qquad j\in\{1,...,N\}\,,%\qquad\text{and}\qquad q_{-1}(t) = \overline{q_1(t)}\,.
    \end{align}
    then its unique solution coincides with the one for the Kuramoto model, i.e.,
    \begin{align*}
        u_j(t) = z_j(t)\,,\qquad j\in\{1,...,N\}\,,
    \end{align*}
    on $t \in [-a, a]$. 
\end{lemma}
\begin{proof}
    Substituting $d_j$, $q_j$ and $\bar{q}_j$ in Eqs.~(\ref{eq:riccati}-\ref{eq:uCI}) gives
    \begin{align}
        \dot{u}_j(t) &= \sum_{k=1}^N A_{jk}z_k(t)  - \sum_{k=1}^N \bar{A}_{jk}\bar{z}_k(t)u_j(t)^2\,, \qquad j\in\{1,...,N\}\label{eq:u1ric}\\
        u_j(0) &= \exp(i\vartheta_j) = z_j(0)\label{eq:u2ric}\,.
    \end{align} 
    Clearly, $u_j(t) = z_j(t)$ for all $j$ and $t$, because $(u_j(t))_j$ is a solution of Eqs.~(\ref{eq:u1ric}-\ref{eq:u2ric}), which are equivalent to Eqs.~(\ref{eq:zkurA}-\ref{eq:z0kurA}) [Lemma~\ref{lem:correspondance_complexe_A}]. Since $f_j(x_j, t) = q_j(t) - \overline{q_j(t)}x_j^2$ for each $j\in\{1, ..., N\}$ is continuous in $t$ and continuously differentiable in $x_j$ for all $x_j\in\mathbb{T}$, then the fundamental existence-uniqueness theorem for non-autonomous systems~\cite[p.77]{Perko2001} guarantees the uniqueness of the solution (recall from Lemma~\ref{lem:correspondance_complexe_A} that $z_1(t)$,...,$z_N(t)$ is also unique).
\end{proof}
\begin{remark}
    One must notice that if either the condition $d_j = \exp(i\vartheta_j)$ or $q_j(t) = \textstyle{\sum_{k=1}^N} A_{jk}z_k(t)$ is not satisfied, then no correspondence exists between the Riccati problem and the Kuramoto model. In perhaps more intuitive terms, different trajectories of the Kuramoto dynamics are related to different Riccati equations (i.e., $q_1(t)$,...,$q_N(t)$ are different for each Kuramoto solution) under the condition that the initial conditions coincide.
\end{remark}
% \begin{lemma}
%     If $d_j \neq \exp{i\vartheta_j}$ for all $j$ while the Riccati problem is still defined with Eq.~\eqref{eq:pj_to_qj}, the unique solution $u_j(t)$ will 
% \end{lemma}

\section{Matone's formula for the projective special unitary group}
\label{subsec:matone}
In 2015, based on recent developments on the Baker-Campbell-Hausdorff equation, Matone found the closed-form equation~\cite[(3.12) with (3.9)]{Matone2015} for the exponential of an element of the Virasoro algebra. Notably, his fundamental result includes the closed-form equation for the exponential of an element of $\mathfrak{sl}_2(\mathbb{C})$, that is, an element of the special linear group $\SL_2(\mathbb{C})$. In the paper, we need a particular case of his result for the automorphisms of the disk, which we present in this subsection. For an arbitrary element of $\SL_2(\mathbb{C})$
\begin{align*}
    \gamma = \begin{pmatrix}\mathcal{A}&\mathcal{B}\\\mathcal{C}&\mathcal{D}\end{pmatrix}\,,\quad \mathcal{A}\mathcal{D}- \mathcal{B}\mathcal{C} = 1\,,
\end{align*}
where $\nu_\pm = s \pm \sqrt{s^2 - 1}$ with $s = \frac{1}{2}\tr\gamma$, Matone's formula is
\begin{align*}
    \exp(a L_{-1})\exp(b L_0)\exp(c L_1) =\exp\left\{\frac{\ln(\nu_+/\nu_{-})}{\nu_+ - \nu_{-}}\left[\mathcal{B}L_{-1} + (\mathcal{A} - \mathcal{D})L_0 -\mathcal{C}L_1\right]\right\}\,,
\end{align*}
where $\mathcal{A} = e^{b/2}(e^{-b} - ac)$, $\mathcal{B} = -ae^{b/2}$, $\mathcal{C} = ce^{b/2}$, $\mathcal{D} = e^{b/2}$ and their inverse is, $a = -\mathcal{B}/\mathcal{D}$, $b = 2\ln \mathcal{D}$, $c = \mathcal{C}/\mathcal{D}$. In that last equation, one also understands $L_{-1}, L_0, L_1$ in their matrix form (abusing notation)
\begin{align*}
    L_{-1} = \begin{pmatrix}
        0 & -1 \\ 0 & 0
    \end{pmatrix}\,, \quad L_0 = \begin{pmatrix}
        -1/2 & 0 \\ 0 & 1/2
    \end{pmatrix}\,, \quad L_1 = \begin{pmatrix}
        0 & 0 \\
        1 & 0
    \end{pmatrix}
    \,.
\end{align*}
% \begin{align*}
%     a = \frac{-\mathcal{B}}{\mathcal{D}}\,,\quad b = 2\ln \mathcal{D}\,,\quad c = \frac{\mathcal{C}}{\mathcal{D}}\,.
% \end{align*}

The automorphism group of the disk is isomorphic to a subgroup of $\PSL_2(\mathbb{C})$. Indeed, $\Aut(\mathbb{D})$ is isomorphic to the projective special unitary group
\begin{align*}
    \PSU(1, 1) = \SU(1, 1)/\{\pm I\} = \left\{ M\{\pm I\}\,|\,M \in \SU(1,1) \,\right\}\,,
\end{align*}
where $\SU(1, 1)$ is the special unitary group, containing $2 \times 2$ complex matrices $M$ with unit determinant satisfying (rather than unitary matrices such that $MM^{\dagger} = I$)
\begin{align}\label{eq:special_unitaire}
    MJM^\dagger = J\,,
\end{align}
where $\dagger$ stands for Hermitian conjugation and 
\begin{align*}
    J = \begin{pmatrix}
        1&0\\0&-1
    \end{pmatrix}\,.
\end{align*}
A disk automorphism can be written as \footnote{Note that this parametrization for the disk automorphism radically simplifies the use of Matone's formula in our experience.}
\begin{align}\label{eq:automorphisme_disque_UV_tz}
    f_{U, V}(z) = \frac{Uz + V}{\overline{V}z + \overline{U}}\,,\quad |U|^2 - |V|^2 = 1\,,
\end{align} 
which is related to a matrix (equivalent to its negative counterpart) of $\PSU(1, 1)$
\begin{align*}
    \begin{pmatrix}U&V\\\bar{V}&\bar{U}\end{pmatrix}
    \,.
\end{align*}
From a hyperbolic geometry perspective, these transformations are the orientation-preserving isometries of the Poincaré disk model. 
% Note personel: orientation-preserving par rapport à l'espace hyperbolique et isométrie au sein du modèle
The matrix group acts on the homogenous coordinates $[z_1:z_2]$ of the complex projective line such that%~\footnote{Les points formant la ligne projective complexe sont définis comme les classes d'équivalence de vecteurs non nulle dans $\mathbb{C}^2$, soit tels que deux vecteurs non nuls $(z_1, z_2)$ et $(z_3, z_4)$ où $z_1,z_2,z_3,z_4 \in\mathbb{C}$ sont équivalents si et seulement si $(z_1, z_2) = (c z_3, c z_4)$ pour un certain $c \in \mathbb{C}\setminus\{0\}$. La classe d'équivalence s'écrit $[z_1:z_2]$, ce qui définit les coordonnées homogènes (ou projectives).} telle que
\begin{align*}
   \begin{pmatrix}U&V\\\bar{V}&\bar{U}\end{pmatrix}\begin{pmatrix}
        z_1\\z_2
    \end{pmatrix} = \begin{pmatrix}
        U z_1 + V z_2\\ \bar{V}z_1 + \bar{U}z_2
    \end{pmatrix}\,.
\end{align*}
Since the homogeneous coordinates $[z_1 :z_2]$ are equivalent to a point $z = z_1/z_2$, then
\begin{align*}    
        [U z_1 + V z_2:\bar{V}z_1 + \bar{U}z_2] \sim \frac{U z_1 + V z_2}{\bar{V}z_1 + \bar{U}z_2} = \frac{U \frac{z_1}{z_2} + V }{\bar{V}\frac{z_1}{z_2} + \bar{U}} = \frac{Uz + V}{\overline{V}z + \overline{U}} = f_{U, V}(z)\,
\end{align*}
as expected. That being said, Matone's formula for $\PSU(1,1)$ is
\begin{align}\label{eq:matoneSI}
    \exp(a L_{-1})\exp(b L_0)\exp(c L_1) =\exp\left\{\nu\left[VL_{-1} + (U - \bar{U})L_0 -\bar{V}L_1\right]\right\}\quad\text{with}\quad \nu = \frac{\ln(\nu_+/\nu_{-})}{\nu_+ - \nu_{-}}\,,
\end{align}
where $\nu_{\pm} = s\pm\sqrt{s^2 - 1}$, $s = \frac{1}{2}(U + \bar{U})$ and $a =-V/\bar{U}$, $b = 2\ln \bar{U}$, $c = \bar{V}/\bar{U}$.
\begin{remark}
    Note that Eq.~(4.7) of Ref.~\cite{Matone2015} and the above equation differ by a sign, which seems to be a minor error in the original reference. %ce qui semble être une coquille dans la référence. Somme toutes, ce signe n'aura pas d'incidence sur les équations déterminantes à la fin, puisque ce sera un ratio $\dot{\nu}/\nu$ qui apparaîtra dans celles-ci.  
\end{remark}

\section{Details on the partial integration of the cross-ratio parts}
\label{subsec:partial_integration}
Below, we provide some additional details about the construction of the partially integrable Kuramoto dynamics on a graph and, notably, the operator-theoretic derivation of the Watanabe-Strogatz transformation using Koopman's perspective~\cite{Koopman1931,Koopman1932}, Magnus expansion~\cite{Magnus1954, Blanes2009}, Matone's formula~\cite{Matone2015} and Watanabe-Strogatz theory itself~\cite{Watanabe1993, Watanabe1994, Pikovsky2008, Marvel2009, Lohe2019}. 

% \vspace{0.3cm}
% \noindent- \textit{Useful partitions}
% \vspace{0.3cm}

First, by Lemma~\ref{lem:correspondance_complexe_A}, the Kuramoto dynamics can be formulated as 
\begin{align*}
    \dot{z}_j = \sum_{k=1}^N A_{jk}z_k - \left(\sum_{k=1}^N \bar{A}_{jk}\bar{z}_k\right)z_j^2 \quad\text{with}\quad A = \frac{1}{2}\left(W \circ e^{-i\alpha} + i\diag(\bm{\omega})\right)\,
\end{align*}
and its Koopman generator is thus
\begin{align*}
    \mathcal{K} = \sum_{j=1}^N\sum_{k=1}^N \left(A_{jk}z_k - \bar{A}_{jk}\bar{z}_kz_j^2\right)\pd[]{}{z_j} = \bm{p}^\top \bm L_{-1} - \bar{\bm{p}}^\top \bm L_1\,.
\end{align*}
% Below, we shall shorten appellations like ``vertices that are part of a motif whose related positions/states $z_j$ are involved in a cross-ratio being a constant of motion'' by ``vertices that are part of a conserved motif''. 
Consider the partition $P = \{\mathcal{M}_1, ..., \mathcal{M}_m, \mathcal{C}_1, ..., \mathcal{C}_c, \mathcal{P}\}$ for the $N$ vertices of a graph $\mathcal{G} = (\mathcal{V}, \mathcal{E})$ with $\mathcal{V} = \{1,...,N\}$ and a (complex) weight matrix $A$, where $\mathcal{P}$ is a set of $p$ vertices that are not part of a conserved motif while the parts $\mathcal{C}_1, ..., \mathcal{C}_c$ admit conserved cross-ratios and the parts $\mathcal{M}_1, ..., \mathcal{M}_m$ admit monomial eigenfunctions. The number of vertices in $\mathcal{C}_\gamma$ is $n_{\gamma} = \# \mathcal{C}_\gamma \geq 4$ for all $\gamma\in\{1,...,c\}$. In such a way, the system is divided into a non-integrable part $\mathcal{P}$ and a partially integrable part $\mathcal{P}_\mathrm{I} = \mathcal{M}_1\cup...\cup\mathcal{M}_m\cup\mathcal{C}_1\cup ...\cup \mathcal{C}_c$. Let us define $\mathcal{M} = \mathcal{M}_1\cup...\cup\mathcal{M}_m$ and $\mathcal{C} = \mathcal{C}_1\cup...\cup\mathcal{C}_c$. We will denote $s$ the surjection that maps a vertex with index $j\in\mathcal{C}$ to the index $\gamma = s(j)$ of its related part $\mathcal{C}_\gamma$. From the latter partition, let's define the coarser partitions of the set $\mathcal{V}$, $\mathcal{Q}_\gamma = \{\mathcal{C}_\gamma,\, Q_{\gamma}\}$ with $Q_{\gamma} = \mathcal{V}\setminus\mathcal{C}_\gamma$ and $P' = \{\mathcal{M}, \mathcal{C}, \mathcal{P}\}$.

Each partially integrable part $\mathcal{C}_1, ..., \mathcal{C}_c$ satisfies the three conditions~\ref{itm:2.1}-\ref{itm:2.3}. First, using $P'$, we split the Koopman generator such as
\begin{align}\label{eq:splitted_generatorSI}
    \mathcal{K} = \sum_{\tau=1}^m \mathcal{J}_\tau + \sum_{\gamma=1}^c\mathcal{K}_\gamma + \mathcal{K}_{\mathrm{NI}}
\end{align}
where 
\begin{align*}
    \mathcal{J}_\tau &= \sum_{\tau = 1}^m\sum_{j\in\mathcal{M}_\tau}\Big(i\omega_jz_j + \frac{1}{2\mu_{\tau j}} \sum_{k\in\mathcal{M}_{\tau}}S_{jk} (z_ke^{-i\kappa_{jk}} - \bar{z}_kz_j^2e^{i\kappa_{jk}})\Big)\pd[]{}{z_j}\\
    \mathcal{K}_\gamma &= \sum_{j\in \mathcal{C}_\gamma}\sum_{k\in\mathcal{V}}
    \left(A_{jk}z_k - \bar{A}_{jk}\bar{z}_kz_j^2\right)\pd[]{}{z_j}\\
    \mathcal{K}_{\mathrm{NI}} &= \sum_{j\in \mathcal{P}}\sum_{k\in\mathcal{V}} \left(A_{jk}z_k - \bar{A}_{jk}\bar{z}_kz_j^2\right)\pd[]{}{z_j}\,.
\end{align*}
% Then, $\mathcal{P}_\mathrm{I}$ can be split into its $c$ different partially integrable parts, giving
% \begin{align*}
%     \mathcal{K} = \sum_{\gamma=1}^c\sum_{j\in \mathcal{C}_\gamma}\sum_{k\in\mathcal{V}}
%     \left(A_{jk}z_k - \bar{A}_{jk}\bar{z}_kz_j^2\right)\pd[]{}{z_j} + \mathcal{K}_{\mathrm{NI}}\,.
% \end{align*}

\vspace{0.3cm}
\noindent- \textit{Explicit application of Theorem~2 from Ref.~\cite{Thibeault2026}}
\vspace{0.3cm}

We can work only with $\mathcal{K}_\gamma$ for now. Using the partition $\mathcal{Q}_{\gamma}$ for the sum over $k$ yields
\begin{align*}
    \mathcal{K}_\gamma = \sum_{j\in \mathcal{C}_\gamma}\left(\sum_{k\in \mathcal{C}_\gamma}
    \left(A_{jk}z_k - \bar{A}_{jk}\bar{z}_kz_j^2\right) + \sum_{k\in Q_{\gamma}}
    \left(A_{jk}z_k - \bar{A}_{jk}\bar{z}_kz_j^2\right)\right)\pd[]{}{z_j}\,.
\end{align*}
By condition~\ref{itm:2.1}, $A_{jk} = A_{s(j)k} =: \mathcal{A}_{\gamma k}$ for all $j\in \mathcal{C}_\gamma$ and $k \in Q_{\gamma}$, where only the elements in the $\gamma$-th row and all columns $k \in Q_{\gamma}$ of the $m\times N$ matrix $\mathcal{A}$ are filled for all $\gamma$, leaving the rest of the elements undefined for now. This leads to
\begin{align*}
    \mathcal{K}_\gamma = \sum_{j\in \mathcal{C}_\gamma}\bigg(\sum_{k\in \mathcal{C}_\gamma}
    \left(A_{jk}z_k - \bar{A}_{jk}\bar{z}_kz_j^2\right) + \sum_{k\in Q_{\gamma}}
    \left(\mathcal{A}_{\gamma k}z_k - \bar{\mathcal{A}}_{\gamma k}\bar{z}_kz_j^2\right)\bigg)\pd[]{}{z_j}\,.
\end{align*}
Similarly, condition~\ref{itm:2.2} gives $A_{jk} =: \mathcal{A}_{\gamma k}$ for all $k \in \mathcal{C}_\gamma\setminus\{j\}$, which fills the elements in the $\gamma$-th row and column $k \in \mathcal{C}_\gamma\setminus\{j\}$ of $\mathcal{A}$ for all $\gamma$. The generator becomes
\begin{align*}
    \mathcal{K}_\gamma = \sum_{j\in \mathcal{C}_\gamma}\bigg((A_{jj} - \bar{A}_{jj})z_j + \sum_{k\in \mathcal{C}_\gamma\setminus\{j\}}\left( \mathcal{A}_{\gamma k}z_k - \bar{\mathcal{A}}_{\gamma k}\bar{z}_kz_j^2\right) + \sum_{k\in Q_{\gamma}}
    \left(\mathcal{A}_{\gamma k}z_k - \bar{\mathcal{A}}_{\gamma k}\bar{z}_kz_j^2\right)\bigg)\pd[]{}{z_j}\,,
\end{align*}
and it can be simplified to
\begin{align*}
    \mathcal{K}_\gamma = \sum_{j\in \mathcal{C}_\gamma}\bigg(i\omega_jz_j + \sum_{k\in \mathcal{V}\setminus\{j\}}\left( \mathcal{A}_{\gamma k}z_k - \bar{\mathcal{A}}_{\gamma k}\bar{z}_kz_j^2\right) \bigg)\pd[]{}{z_j}\,.
\end{align*}
Finally, condition~\ref{itm:2.3} is equivalent to 
\begin{align*}
    \omega_j = \omega_{\ell_{\gamma}} + 2\Imag(\mathcal{A}_{\gamma j} - \mathcal{A}_{\gamma \ell_{\gamma}})\,,
\end{align*}
where $\ell_{\gamma}\in \mathcal{C}_\gamma$ labels an arbitrary ``reference oscillator'' within $\mathcal{C}_\gamma$. Consequently, 
\begin{align*}
    A_{jj} = \frac{i\omega_j}{2} = \frac{i}{2}(\omega_{\ell_{\gamma}} + 2\Imag(\mathcal{A}_{\gamma j} - \mathcal{A}_{\gamma \ell_{\gamma}}))%\,,\quad j\in \mathcal{C}_\gamma
\end{align*}
for some $\mathcal{A}_{\gamma j} \in\mathbb{C}$ (but only its imaginary part contributes), thus completely fixing $\mathcal{A}$. Applying the condition for any $j \in \mathcal{C}_\gamma$ leads to
\begin{align*}
    \mathcal{K}_\gamma = \sum_{j\in \mathcal{C}_\gamma}\bigg(i (\omega_{\ell_{\gamma}} + 2\Imag(\mathcal{A}_{\gamma j} - \mathcal{A}_{\gamma \ell_{\gamma}}))z_j + \sum_{k\in \mathcal{V}\setminus\{j\}}\left( \mathcal{A}_{\gamma k}z_k - \bar{\mathcal{A}}_{\gamma k}\bar{z}_kz_j^2\right) \bigg)\pd[]{}{z_j}\,.
\end{align*}
Yet, $2i \Imag(\mathcal{A}_{\gamma j})z_j = \mathcal{A}_{\gamma j}z_j - \bar{\mathcal{A}}_{\gamma j}z_j = \mathcal{A}_{\gamma j}z_j - \bar{\mathcal{A}}_{\gamma j}\bar{z}_jz_j^2$ and therefore
\begin{align*}
    \mathcal{K}_\gamma = \sum_{j\in \mathcal{C}_\gamma}\bigg(i(\omega_{\ell_{\gamma}} - 2\Imag(\mathcal{A}_{\gamma \ell_{\gamma}}))z_j + \sum_{k\in \mathcal{V}}\left( \mathcal{A}_{\gamma k}z_k - \bar{\mathcal{A}}_{\gamma k}\bar{z}_kz_j^2\right) \bigg)\pd[]{}{z_j}\,.
\end{align*}
This finally leads to the more elegant form
\begin{align}\label{eq:kur_PmuSI}
     \mathcal{K}_{\gamma} &=  \rho_\gamma(z) L_{-1}^{\gamma} + i \Omega_{\gamma} L_0^{\gamma} - \overline{\rho_\gamma(z)} L_1^{\gamma}
\end{align}
with 
\begin{align*}
    \Omega_{\gamma} =  \omega_{\ell_{\gamma}} - 2\Imag(\mathcal{A}_{\gamma \ell_{\gamma}})\,,\qquad \rho_{\gamma}(z) = \sum_{k=1}^N \mathcal{A}_{\gamma k}z_k\,,\qquad L_n^{\gamma} = \sum_{j\in \mathcal{C}_\gamma} z_j^{n+1} \pd[]{}{z_j}\,.
\end{align*}
Note that under this form, we observe that $\Omega_{\gamma}$ acts as a new effective natural frequency for the oscillators in $\mathcal{C}_\gamma$ that depends on the original natural frequencies, the phase lags and the weight matrix. One can now clearly notice that the cross-ratios related to each partially integrable parts are conserved quantities, since they are the joint invariants of $L_{-1}^\gamma, L_0^\gamma, L_1^\gamma$ for each $\gamma$. Also, $\mathcal{K}$ in Eq.~\eqref{eq:kur_PmuSI} thus corresponds to the Koopman generator of 
\begin{align}
     %\dot{z}_j(t) &=  \sum_{k=1}^N \mathcal{A}_{s(j) k}z_k(t) + i \Omega_{s(j)} z_j(t) - \left(\sum_{k=1}^N\bar{\mathcal{A}}_{s(j) k}\bar{z}_k(t)\right)z_j(t)^2\,,\hspace{-3cm} &&j\in  P_{\mathrm{I}}  \label{eq:kur_PI}\\ 
      \dot{z}_j(t) &=  \sum_{k=1}^N \mathcal{A}_{\gamma k}z_k(t) + i \Omega_{\gamma} z_j(t) - \left(\sum_{k=1}^N\bar{\mathcal{A}}_{\gamma k}\bar{z}_k(t)\right)z_j(t)^2\,,\hspace{-1.5cm} &&\forall j\in  \mathcal{C}_\gamma\,,\label{eq:kur_PI} %\,\,\forall \gamma \in\{1,...,c\} \label{eq:kur_PI}%\\
     % \dot{z}_\ell(t) &=  \sum_{k=1}^N A_{\ell k}z_k(t) - \left(\sum_{k=1}^N\bar{A}_{\ell k}\bar{z}_k(t)\right)z_\ell(t)^2\,,\hspace{-3cm} &&\ell \in  \mathcal{P} \label{eq:kur_P0}
\end{align}
with initial conditions
\begin{align}
    z_1(0) = \exp(i\vartheta_1)\,,\quad ...\quad, \quad z_N(0) = \exp(i\vartheta_N) \label{eq:kur_CI}
\end{align}
for $\vartheta_1,...,\vartheta_N \in \mathbb{R}$. % Note that it is not necessary to go back to the differential equations as above and below (see main text).

\vspace{0.3cm}
\noindent- \textit{Transforming the partially integrable equations to Riccati equations}
\vspace{0.3cm}

%Lemma~\ref{lem:correspondance_complexe_A} guarantees the unicity of the solution to problem~(\ref{eq:kur_PI}-\ref{eq:kur_CI}). 
For the given initial condition~\eqref{eq:kur_CI}, one can proceed as in Lemma~\ref{lem:correspondance_nonautonome} and relate the solution $(z_j(t))_{j\in \mathcal{C}_\gamma}$ of Eq.~\eqref{eq:kur_PmuSI} for each $\gamma\in\{1,...,c\}$ to the solution of the Riccati equations
\begin{align*}
    \dot{u}_j(t) &=  q_{\gamma}(t) + i \Omega_{\gamma} u_j(t) - \overline{q_{\gamma}(t)}u_j(t)^2\,,\quad u_j(0) = z_j(0)\,,\quad j\in \mathcal{C}_\gamma\,,
\end{align*}
with 
\begin{align*}
    q_{\gamma}(t) &= \rho_{\gamma}(z(t)) = \sum_{k=1}^N \mathcal{A}_{\gamma k}z_k(t)\,.
\end{align*}
Note that there are \textit{different} sets (since, generally, $q_{\gamma}(t) \neq q_{\nu}(t)$, $\gamma \neq \nu$) of \textit{identical} Riccati equations (although the initial condition for each of these identical Riccati equations varies in general). The Koopman generator of the above Riccati equations, for all $\gamma$, is 
\begin{align}\label{eq:generator_ricSI}
    \mathcal{R}_{\gamma}(t, z) = q_{\gamma}(t) L_{-1}^{\gamma} + i\Omega_\gamma L_0^{\gamma} - \overline{q_{\gamma}(t)}L_1^{\gamma}\,,
\end{align}
which is a tangent vector at $(t, z) := (t, z_1,...,z_N)$ with null component $\pd[]{}{t}$ or more simply, a tangent vector at $(z_1,...,z_N)$ with time-varying coefficients. For the sake of the next argument, let's write the Koopman generator $\mathcal{K}$ as $\mathcal{K}(z)$ (a tangent vector at $z_1,...,z_N$) to be more precise. From a differential geometry perspective, the vector field $w_\mathcal{K}$ along the solution curve $(z_j(t))_{j\in \mathcal{C}_\gamma}$ is the same as the vector field $w_\mathcal{R}$ along this curve, meaning that locally $\mathcal{K}_{\gamma}(z(t)) = \mathcal{R}_{\gamma}(t, z(t))$. In the following, we will write $\mathcal{R}_{\gamma}(t)$ rather than $\mathcal{R}_{\gamma}(t, z(t))$ to simplify the notation while insisting on the time dependence.

\vspace{0.3cm}
\noindent- \textit{Observables dynamics}
\vspace{0.3cm}

Under Koopman's perspective and a complexity science perspective (as argued in the paper, both are aligned), the natural approach is to consider the dynamics of the observables depending on the position of the oscillators, which include synchronization observables. The generator of the time evolution for these quantities is precisely the Koopman generator. In our case, the Koopman generators act on a space of functions on the $N$-torus, i.e.,
\begin{align*}
    f:\,\mathbb{T}^N \to \mathbb{C}\,.
\end{align*}
Among all the possible functions, we will be interested in the observables $\{g_{\gamma}\}_{\gamma=\{1,...,c\}}$ defined such that 
\begin{equation*}
\begin{tikzcd}
\mathbb{T}^N\drar["g_{\gamma}"]\dar["r_{\gamma}" left]  &  &
(z_j)_{j\in \mathcal{V}}\drar[maps to, "g_{\gamma}"] \dar["r_{\gamma}" left, maps to] & \\
\mathbb{T}^{n_{\gamma}} \rar[swap, "f_{\gamma}"] & \mathbb{C} &
(z_j)_{j\in \mathcal{C}_\gamma} \rar[maps to, swap, "f_{\gamma}"] &  c
\end{tikzcd}
\end{equation*}
The dynamics of these observables is 
\begin{align}\label{eq:observable_dyn}
    \dot{g}_\gamma(z(t)) = \mathcal{R}_{\gamma}(t)g_\gamma(z(t))\,,\qquad \forall \gamma\in\{1,...,c\}\,,
\end{align}
which are \textit{non-autonomous, linear, and uncoupled} differential equations. There at least two ways to approach such equations, but one turns out to be particularly useful.

\vspace{0.3cm}
\noindent- \textit{Magnus expansion}
\vspace{0.3cm}

To solve Eq.~\eqref{eq:observable_dyn} for a given $\gamma$, we use the Magnus expansion~\cite{Magnus1954, Blanes2009}. Using Theorem~4 in Ref.~\cite{Blanes2009} yields%and considering that $s=0 \leq t$ is the initial time~\footnote{Recall that the flot of a non-autonomous system really depends over two time parameters. Yet, setting the initial time to $s = 0$ is convenient and is done without loss of generality.} yields
\begin{align*}
    g_{\gamma}(z(t)) = \exp\left(\mathcal{L}_{\gamma}(t)\right)g_{\gamma}(z(0))\quad\text{with}\quad \mathcal{L}_{\gamma}(t) = \sum_{n=1}^\infty \mathcal{L}_n^\gamma(t)\,,
\end{align*}
where the first three terms of the expansion are
\begin{align*}
\mathcal{L}_1^\gamma(t) &= \int_{0}^{t} \mathcal{R}_{\gamma}(t_1) \, \dif t_1\\
\mathcal{L}_2^\gamma(t) &= \frac{1}{2} \int_{0}^{t} \int_{0}^{t_1} [\mathcal{R}_{\gamma}(t_1), \mathcal{R}_{\gamma}(t_2)] \, \dif t_2 \, \dif t_1\\
\mathcal{L}_3^\gamma(t) &= \frac{1}{6} \int_{0}^{t} \int_{0}^{t_1} \int_{0}^{t_2} \left( [\mathcal{R}_{\gamma}(t_1), [\mathcal{R}_{\gamma}(t_2), \mathcal{R}_{\gamma}(t_3)]] + [\mathcal{R}_{\gamma}(t_3), [\mathcal{R}_{\gamma}(t_2), \mathcal{R}_{\gamma}(t_1)]] \right) \dif t_3 \, \dif t_2 \, \dif t_1\,.
\end{align*}
%In the following, \textbf{we drop every index} $\gamma$ to simplify the notation, but it should be kept in mind that we perform these operations for each partially integrable parts individually. 
With Eq.~\eqref{eq:generator_ricSI}, the first term becomes
\begin{align*}
    \mathcal{L}_1^{\gamma}(t) %&= \left(\int_{t}^{s} q_1(s_1)\, \dif s_1 \right) L_{-1} + i\omega (s-t) L_0 - \left(\int_{t}^{s}q_{-1}(s_1) \, \dif s_1\right)L_1\\
    &= b_1^{\gamma}(t)L_{-1}^{\gamma} + 2iy_1^{\gamma}(t)L_0^{\gamma} - \overline{b_1^{\gamma}(t)}L_1^{\gamma}\,,
\end{align*}
where
\begin{align*}
    b_1^{\gamma}(t) = \int_{0}^{t} q_{\gamma}(t_1)\, \dif t_1\qquad \text{and} \qquad y_1^{\gamma}(t) = \frac{1}{2}\Omega_{\gamma} t\,.
\end{align*}
Then, the second term is 
\begin{align*}
    \mathcal{L}_2^{\gamma}(t) &=\frac{1}{2} \int_{0}^{t} \int_{0}^{t_1} [q_{\gamma}(t_1) L_{-1}^{\gamma} + i\Omega_{\gamma} L_0^{\gamma} - \overline{q_{\gamma}(t_1)} L_1^{\gamma}\,\,, \,\,q_{\gamma}(t_2) L_{-1}^{\gamma} + i\Omega_{\gamma} L_0^{\gamma} - \overline{q_{\gamma}(t_2)} L_1^{\gamma}] \, \dif t_2 \, \dif t_1
    \\& = \frac{1}{2} \int_{0}^{t} \int_{0}^{t_1} (i\Omega_{\gamma}(q_{\gamma}(t_1) - q_{\gamma}(t_2)) L_{-1} + 2 (\overline{q_{\gamma}(t_1)}q_{\gamma}(t_2) -  q_{\gamma}(t_1) \overline{q_{\gamma}(t_2)})L_0^{\gamma} -(- i\Omega_{\gamma}(\overline{q_{\gamma}(t_1)} - \overline{q_{\gamma}(t_2)})) L_1^{\gamma}) \, \dif t_2 \, \dif t_1\\
    &= b_2^{\gamma}(t)L_{-1}^{\gamma} + 2iy_2^{\gamma}(t)L_0^{\gamma} - \overline{b_2^{\gamma}(t)}L_1^{\gamma}
\end{align*}
with
\begin{align*}
    b_2^{\gamma}(t) = \frac{i\Omega_{\gamma}}{2}\int_{0}^{t} \int_{0}^{t_1}(q_{\gamma}(t_1)-q_{\gamma}(t_2)) \, \dif t_2 \, \dif t_1\qquad \text{and} \qquad
    y_2^{\gamma}(t) = \frac{1}{2i}\int_{0}^{t} \int_{0}^{t_1}(\overline{q_{\gamma}(t_1)}q_{\gamma}(t_2) -  q_{\gamma}(t_1) \overline{q_{\gamma}(t_2)})\, \dif t_2 \, \dif t_1\,.
\end{align*}
Ultimately, since the integrals only act on the coefficients of the operators, we have 
\begin{align*}
    \mathcal{L}_n^{\gamma}(t) = b_n^{\gamma}(t)L_{-1}^{\gamma} + 2iy_n^{\gamma}(t) L_0^{\gamma} - \overline{b_n^{\gamma}(t)} L_1^{\gamma}\,,
\end{align*}
which means that $\mathcal{L}(t)$ also has the expected form (see Ref.~\cite[p.162]{Blanes2009}) of an element of the Lie algebra of $\Aut(\mathbb{D})$:
\begin{align*}
    \mathcal{L}_{\gamma}(t) = B_{\gamma}(t)L_{-1} + 2iY_{\gamma}(t)L_0 - \overline{B_{\gamma}(t)}L_1\quad\text{with}\quad B_{\gamma}(t) = \sum_{n=1}^\infty b_n^{\gamma}(t)\,,\,\,Y_{\gamma}(t) = \sum_{n=1}^\infty y_n^{\gamma}(t)\,,
\end{align*}
and hence
\begin{align}\label{eq:g_exp}
    g_\gamma(z(t)) = \exp(B_{\gamma}(t)L_{-1} + 2iY_{\gamma}(t)L_0 - \overline{B_{\gamma}(t)}L_1)g_\gamma(z(0))\,.
\end{align}
We now aim to find the explicit form of the exponential map of the operator $B_{\gamma}(t)L_{-1} + 2iY_{\gamma}(t)L_0 - \overline{B_{\gamma}(t)}L_1$. % , which can be associated to a matrix in the matrix group $\PSU(1, 1)$ as we shall see.

\vspace{0.3cm}
\noindent- \textit{Matone's formula}
\vspace{0.3cm}

On this subject, a breakthrough was made by Matone in 2015~\cite{Matone2015}. As presented in subsection~\ref{subsec:matone}, one can adapt his result for $\PSU(1, 1)$ (Eq.~\eqref{eq:matoneSI})\,, and the exponential map in Eq.~\eqref{eq:g_exp} is only a few steps away from a direct application of Matone's formula.

To use the formula in our context, first define
\begin{align*}
    U_{\gamma}(t) := X_{\gamma} + \frac{i Y_{\gamma}(t)}{\nu(X_{\gamma})}\,,\quad
    V_{\gamma}(t) := \frac{B_{\gamma}(t)}{\nu(X_{\gamma})}\quad \text{with}\quad \nu(x) := \frac{\ln\left(\frac{x + \sqrt{x^2 - 1}}{x - \sqrt{x^2 -1}}\right)}{2\sqrt{x^2 - 1}} \in \mathbb{R}\,,
\end{align*}
where $X_{\gamma}$ is constrained by the hyperbolic condition
\begin{align}\label{eq:contrainte_hyperbolique_UVst}
    |U_{\gamma}(t)|^2 - |V_{\gamma}(t)|^2 = 1
\end{align}
for all $t$. The latter constraint has the more explicit form
\begin{align*}
    X_{\gamma}^2 + \frac{Y_{\gamma}(t)^2}{\nu(X_{\gamma})^2} - \frac{R_{\gamma}(t)^2}{\nu(X_{\gamma})^2} - 1 &= 0\,,
\end{align*}
considering that $V_{\gamma}(t) = (R_{\gamma}(t)/\nu(X_{\gamma}))\exp(i\Phi_{\gamma}(t))$ where $R_{\gamma}(t) = |B_{\gamma}(t)|$.
\begin{figure}[b]
    \centering
    \includegraphics[width=0.6\linewidth]{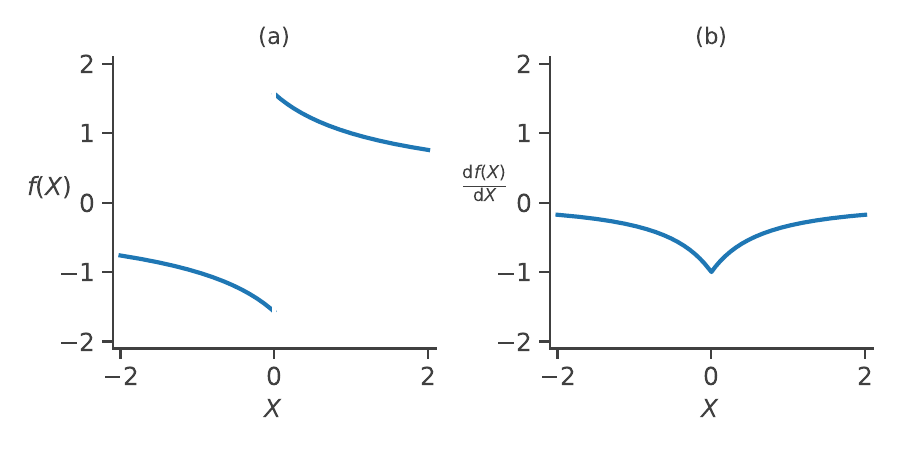}
    \caption{(a) Principal branch ($k = 0$) of $\nu(X) = f(X)$ with respect to $X$ and (b) its derivative.}
    \label{fig:nu_vs_X2}
\end{figure}
Since $\nu(X_{\gamma}) \neq 0$ for all $X_{\gamma}$ (see Fig.~\ref{fig:nu_vs_X2}), the constraint becomes
\begin{align*}
     (X_{\gamma}^2 - 1)\nu(X_{\gamma})^2 + Y_{\gamma}(t)^2 - R_{\gamma}(t)^2 &= 0\,.
\end{align*}
A priori, it is not clear if there exists a $X_{\gamma}$ that satisfies this equation. However,
\begin{align*}
    (X_{\gamma}^2 - 1)\nu(X_{\gamma})^2 = (X_{\gamma}^2 - 1)\frac{\ln^2\left(\frac{X_{\gamma} + \sqrt{X_{\gamma}^2 - 1}}{X_{\gamma} - \sqrt{X_{\gamma}^2 -1}}\right)}{4(X_{\gamma}^2 - 1)} = \frac{1}{4}\ln^2\left(\frac{X_{\gamma} + \sqrt{X_{\gamma}^2 - 1}}{X_{\gamma} - \sqrt{X_{\gamma}^2 -1}}\right) 
\end{align*}
and therefore, 
\begin{align*}
     \frac{1}{4}\ln^2\left(\frac{X_{\gamma} + \sqrt{X_{\gamma}^2 - 1}}{X_{\gamma} - \sqrt{X_{\gamma}^2 -1}}\right) + Y_{\gamma}(t)^2 - R_{\gamma}(t)^2 &= 0\,.
\end{align*}
The inversion of the last equation gives a much more elegant relation, i.e.,
\begin{align*}
    X_{\gamma} &= \pm \cosh(\sqrt{R_{\gamma}(t)^2 - Y_{\gamma}(t)^2}) = \left\{
\begin{aligned}
    \pm\cosh(\sqrt{|R_{\gamma}(t)^2 - Y_{\gamma}(t)^2|}) &\,\,\text{ if }\,\, R_{\gamma}(t)^2 - Y_{\gamma}(t)^2  \geq 0 \\
    \pm \cos(\sqrt{|R_{\gamma}(t)^2 - Y_{\gamma}(t)^2|}) &\,\,\text{ if }\,\,R_{\gamma}(t)^2 - Y_{\gamma}(t)^2  < 0
\end{aligned}
\right. \,\,.
%     \\&= \pm\cosh(\sqrt{|R_{\gamma}(t)^2 - Y_{\gamma}(t)^2|})\,\delta_{R_{\gamma}(t)^2 - Y_{\gamma}(t)^2  \geq 0} \pm \cos(\sqrt{|R_{\gamma}(t)^2 - Y_{\gamma}(t)^2|})\,\delta_{R_{\gamma}(t)^2 - Y_{\gamma}(t)^2  < 0}
\end{align*}
This demonstrates that for all $R_{\gamma}(t), Y_{\gamma}(t)$, there exists at least one real value $X_{\gamma}$ that satisfies the constraints. % (see the \href{https://www.desmos.com/calculator/1wcgz3thvz?lang=fr}{plot}). % and the verification done with \href{https://www.wolframalpha.com/input?i=inverse++%5Cfrac%7B1%7D%7B4%7D%5Cln%5E2%5Cleft%28%5Cfrac%7BX+%2B+%5Csqrt%7BX%5E2+-+1%7D%7D%7BX+-+%5Csqrt%7BX%5E2+-1%7D%7D%5Cright%29}{Wolfram Alpha}).
% ou bien, en réarrangeant,
% \begin{align*}
%     \frac{X}{\sqrt{X^2 - 1}} = -\,\frac{1+\exp(\pm 2\sqrt{R(t)^2 - Y(t)^2})}{1-\exp(\pm 2\sqrt{R(t)^2 - Y(t)^2})}
% \end{align*}
Altogether, one gets 
\begin{align*}
    g_{\gamma}(z(t)) = \exp[\nu(X_{\gamma})(V_{\gamma}(t)L_{-1} + (U_{\gamma}(t) - \overline{U_{\gamma}(t)})L_0 - \overline{V_{\gamma}(t)}L_1)]g_{\gamma}(z(0))
\end{align*}
and from Matone's formula, the time evolution of $g_{\gamma}$ is given by the disk automorphism with time-varying coefficients
\begin{align*}
    g_{\gamma}(z(t)) = K_{\phi_{0, t}}^{\mathrm{Ric}}[f](z(0)) = \frac{U_{\gamma}(t)g_{\gamma}(z(0)) + V_{\gamma}(t)}{\overline{V_{\gamma}(t)}g_{\gamma}(z(0)) + \overline{U_{\gamma}(t)}}\,,
\end{align*}
where $K^{\mathrm{Ric}}$ is the Koopman operator of the Riccati dynamics. In particular, the choice of observable $g_\gamma(z_1,...,z_N) = z_j$ for $j\in \mathcal{C}_\gamma$ implies that the solution for the Riccati dynamics (the one related to the Kuramoto model at the coinciding initial condition) is 
\begin{align}\label{eq:sol_gen_ricSI}
    z_j(t) = \phi_{0,t}^{\mathrm{Ric}}(z_j(0)) = \frac{U_{\gamma}(t)z_j(0) + V_{\gamma}(t)}{\overline{V_{\gamma}(t)}z_j(0) + \overline{U_{\gamma}(t)}}\,,\quad \forall j\in \mathcal{C}_\gamma\,,
\end{align}
where %we have reintroduced the indices $\gamma$ where they should be and 
$U_{\gamma}(t), V_{\gamma}(t)$ satisfy Eq.~\eqref{eq:contrainte_hyperbolique_UVst} for all $t$ and $\gamma$. 

There are infinitely many integrals to be solved to get $U_{\gamma}(t)$ and $V_{\gamma}(t)$, but the good news is that there is an alternative: interpreting them as a solution to a set of fewer differential equations than the original system. In such a way, instead of computing the integrals, one aims to find these reduced differential equations. This idea essentially goes back to the original theory presented by Watanabe and Strogatz in 1994~\cite{Watanabe1994}, but one can already observe that we did not (and will not) make any Ansatz whatsoever.

\vspace{0.3cm}
\noindent- \textit{Deducing Watanabe-Strogatz transformation}
\vspace{0.3cm}

Under this perspective, at $t = 0$, every solution $(U_{\gamma}(t), V_{\gamma}(t))$ of the yet-to-be-determined differential equations is constrained to start at $(U_{\gamma}(0), V_{\gamma}(0)) = (\pm 1, 0)$  \textit{for any initial conditions} $(z_j(0))_{j\in \mathcal{C}_\gamma} =: (\xi_j)_{j\in \mathcal{C}_\gamma} \in \mathbb{T}^{n_{\gamma}}$ in order for Eq.~\eqref{eq:sol_gen_ricSI} to be satisfied. It would however be more informative to set things in such a way that different initial conditions $(\xi_j)_{j\in \mathcal{C}_\gamma}$ of the Kuramoto model~\footnote{For different initial conditions, it is the trajectories in different Riccati dynamics that are related to the trajectories of the Kuramoto model. Also, one must remember that the goal is to come back to the Kuramoto model at some point.} leading to different trajectories are associated to different initial conditions of the disk automorphism coefficients (a consideration in this regard is presented in Ref.~\cite[section 4.2.1]{Watanabe1994}). To do so, one can express the initial conditions as
\begin{align*}
    \xi_j = \frac{a_{\gamma}w_j + b_{\gamma}}{\overline{b_{\gamma}}w_j + \overline{a
    _{\gamma}}}\,,\quad |a_{\gamma}|^2 - |b_{\gamma}|^2 = 1\,,
\end{align*}
where $w_j = e^{i\psi_j}$ ($\psi$ is not to be confused with others in the paper) for all $j\in \mathcal{C}_\gamma$, for all $\gamma\in\{1,...,c\}$ \footnote{We let $w_j$ for $j\in \mathcal{P}$ undefined (they are not used).}. Consequently,
\begin{align}\label{eq:ws_UVSI}
    z_j(t) = \frac{u_{\gamma}(t)w_j + v_{\gamma}(t)}{\overline{v_{\gamma}(t)}w_j + \overline{u_{\gamma}(t)}}\,,\quad \forall j\in \mathcal{C}_\gamma\,,
\end{align}
with
\begin{align*}
    u_{\gamma}(t) = a_{\gamma}U_{\gamma}(t) + \overline{b}_{\gamma}V_{\gamma}(t)\quad \text{and} \quad v_{\gamma}(t) = b_{\gamma}U_{\gamma}(t) + \overline{a}_{\gamma}V_{\gamma}(t)\,,
\end{align*}
where it is easy to demonstrate that $|u_{\gamma}(t)|^2 - |v_{\gamma}(t)|^2 = 1$. 

Moreover, $u_{\gamma}(t)$ and $v_{\gamma}(t)$ are not bounded and since we aim to find differential equations, it is convenient to make the change of coordinates
\begin{align*}
    u_{\gamma}(t) = \frac{\zeta_{\gamma}(t)^{1/2}}{\sqrt{1 - |Z_{\gamma}(t)|^2}}\,,\qquad v_{\gamma}(t) = \frac{\zeta_{\gamma}(t)^{-1/2}Z_{\gamma}(t)}{\sqrt{1 - |Z_{\gamma}(t)|^2}}\,,
\end{align*}
for the disk and the unit circle, where $\zeta_{\gamma}(t) = e^{i\varphi_\gamma(t)}$ and $Z_{\gamma}(t) \in \mathbb{D}$. This change of coordinates applied to Eq.~\eqref{eq:ws_UVSI} leads to 
\begin{align}\label{eq:auto_ricSI}
    z_j(t) = M_{Z_{\gamma}(t), \zeta_{\gamma}(t)}(w_j) = \frac{\zeta_{\gamma}(t)w_j + Z_{\gamma}(t)}{1 + \zeta_{\gamma}(t)\,\overline{Z_{\gamma}(t)}w_j}\quad \mathrm{with}\quad \zeta_{\gamma}(t) = e^{i\varphi_{\gamma}(t)},\, |Z_{\gamma}(t)| < 1\,,j\in \mathcal{C}_\gamma\,
\end{align}
for all $t$ and we have therefore deduced the Watanabe-Strogatz transformation (the transformation that leads to the real, original form~\cite{Watanabe1994} is straightforward to obtain by using trigonometric identities~\cite[IV, A]{Marvel2009}). 

From there, we observe that the introduction of $(w_j)_{j\in \mathcal{C}_\gamma} = (e^{i\psi_j})_{j\in \mathcal{C}_\gamma}$ gives the desired freedom, because Eq.~\eqref{eq:auto_ricSI} for all $j \in \mathcal{C}_\gamma$ gives a set of $n_{\gamma}$ constraints and there are $n_{\gamma} + 3$ real parameters:
\begin{align*}
    (\psi_j)_{j\in \mathcal{C}_\gamma}\,,\quad|Z_{\gamma}(0)|\,,\quad \arg(Z_{\gamma}(0))\,,\quad \varphi_{\gamma}(0)\,.
\end{align*}
To set the values of all these parameters, three other conditions can be added on the phases $(\psi_j)_{j\in \mathcal{C}_\gamma}$ in such a way that the initial conditions $z_j(0)=\xi_j$ for all $j \in \mathcal{C}_\gamma$ are associated to initial conditions $Z_{\gamma}(0), \varphi_{\gamma}(0)$. We refer the reader to the original work of Watanabe-Stroatz~\cite{Watanabe1994} and to the code in Ref.~\cite{Thibeault2026_koopman_kuramoto} for more details on how to fix the initial conditions. The rest of derivation is done in the main text. We shall now provide two concrete examples for the partial integration of the cross-ratio parts.

\section{Partial integration of basic case studies with conserved cross-ratios}
\label{subsec:basic_cases}
To begin with, we address the case of one partially integrable part of 4 vertices and one non-integrable part of one vertex.
\begin{example}
    Consider a graph of $N= 5$ vertices partitioned into $\{\mathcal{C}_1, \mathcal{P}\}$ with $\mathcal{C}_1 = \{1,2,3,4\}$, $\mathcal{P} = \{5\}$, and with complex weight matrix 
    \begin{align*}
    A = \begin{pmatrix}
        i\omega_1/2 & \mathcal{A}_2 & \mathcal{A}_3 & \mathcal{A}_4 & \mathcal{A}_5\\
        \mathcal{A}_1 & i\omega_2 /2 & \mathcal{A}_3 & \mathcal{A}_4 & \mathcal{A}_5\\
        \mathcal{A}_1 & \mathcal{A}_2 & i\omega_3/2 & \mathcal{A}_4 & \mathcal{A}_5\\
        \mathcal{A}_1 & \mathcal{A}_2 & \mathcal{A}_3 & i\omega_4/2 & \mathcal{A}_5\\
        A_{51} & A_{52} & A_{53} & A_{54} & i\omega/2
    \end{pmatrix}\,,
\end{align*}
where $\omega_1$ is fixed to some real value (the first oscillator is the reference oscillator)  and 
\begin{align*}
    \omega_j = \omega_1 + 2\Imag(\mathcal{A}_j - \mathcal{A}_1)\,,\quad j\in \{2,3,4\}\,.
\end{align*}
By construction, the effective natural frequency of each oscillator in $\mathcal{C}_1$ is $\Omega = \omega_1 - 2\Imag(\mathcal{A}_1)$ while the natural frequency of the fifth oscillator in $\mathcal{P}$ is $\omega$. In such a case, $c=1$ and $p = 1$, implying that the Kuramoto dynamics can be reduced to $n = 3c + p = 4$ real differential equations.
These equations are, in complex form,
\begin{align*}  
    \dot{Z} &= i\Omega Z + F(Z, \zeta, z_5) - \overline{F(Z, \zeta, z_5)}Z^2\,,\\
    \dot{\zeta} &= (i\Omega + F(Z, \zeta, z_5)\overline{Z} - \overline{F(Z, \zeta, z_5)}\,Z)\,\zeta\,,\\
    \dot{z}_5 &= i\omega z_5 + G(Z, \zeta;\bm{w}) - \overline{G(Z, \zeta;\bm{w})}z_5^2\,,
\end{align*}
where we have used $\sum_{k\in \mathcal{P}} (A_{5 k}z_k - \bar{A}_{5 k}\overline{z_k}z_5^2) = (A_{55} - \bar{A}_{55})z_5 = i\omega z_5$ and
\begin{align*}
    M_{Z, \zeta }(w_j) = \frac{\zeta w_j + Z}{1 + \zeta\,\overline{Z} w_j}\,,\quad G(Z, \zeta;\bm{w}) = \sum_{k=1}^4 A_{5 k}M_{Z, \zeta}(w_k)\,,\qquad
    F(Z, \zeta, \bm{z}_{\mathcal{P}}) = \mathcal{A}_{5}z_5 + \sum_{k=1}^4 \mathcal{A}_{k}M_{Z, \zeta}(w_k)\,.
\end{align*}
We also have $n_1 = 4$ and there is $q = N - n = n_1 - 3 = 1$ constant of motion in the system: there is only one functionally independent conserved cross-ratio, say
\begin{align*}
    c_{1234}(z) = \frac{(z_{3}-z_{1})(z_{4}-z_{2})}{(z_{3}-z_{2})(z_{4}-z_{1})}\,.
\end{align*}
\end{example}
Let's now present a more general example in terms of network structure.
\begin{example}
    Consider a graph with $N = 13$ vertices partitioned into the parts $\mathcal{P} = \{11, 12, 13\}$ and $\mathcal{C}_1 = \{1,2,3,4\}$, $\mathcal{C}_2 = \{5,6,7,8,9,10\}$. Then, we can construct the general complex weight matrix
    \begin{equation*}
        A = \begin{pmatrix}
        i\omega_1/2 & \mathcal{A}_{1,2} & \mathcal{A}_{1,3} & \mathcal{A}_{1,4} & \mathcal{A}_{1,5} & \mathcal{A}_{1,6}& \mathcal{A}_{1,7} & \mathcal{A}_{1,8}& \mathcal{A}_{1,9} & \mathcal{A}_{1,10} & \mathcal{A}_{1,11} & \mathcal{A}_{1,12} & \mathcal{A}_{1,13}\\
        \mathcal{A}_{1,1} & i\omega_2/2 & \mathcal{A}_{1,3} & \mathcal{A}_{1,4} & \mathcal{A}_{1,5} & \mathcal{A}_{1,6}& \mathcal{A}_{1,7} & \mathcal{A}_{1,8}& \mathcal{A}_{1,9} & \mathcal{A}_{1,10} & \mathcal{A}_{1,11} & \mathcal{A}_{1,12} & \mathcal{A}_{1,13}\\
        \mathcal{A}_{1,1} & \mathcal{A}_{1,2} & i\omega_3/2 & \mathcal{A}_{1,4} & \mathcal{A}_{1,5} & \mathcal{A}_{1,6}& \mathcal{A}_{1,7} & \mathcal{A}_{1,8}& \mathcal{A}_{1,9} & \mathcal{A}_{1,10} & \mathcal{A}_{1,11} & \mathcal{A}_{1,12} & \mathcal{A}_{1,13}\\
        \mathcal{A}_{1,1} & \mathcal{A}_{1,2} & \mathcal{A}_{1,3} & i\omega_4/2 & \mathcal{A}_{1,5} & \mathcal{A}_{1,6}& \mathcal{A}_{1,7} & \mathcal{A}_{1,8}& \mathcal{A}_{1,9} & \mathcal{A}_{1,10} & \mathcal{A}_{1,11} & \mathcal{A}_{1,12} & \mathcal{A}_{1,13}\\
        \hline 
        \mathcal{A}_{2,1} & \mathcal{A}_{2,2} & \mathcal{A}_{2,3} & \mathcal{A}_{2,4} & i\omega_5/2 & \mathcal{A}_{2,6}& \mathcal{A}_{2,7} & \mathcal{A}_{2,8}& \mathcal{A}_{2,9} & \mathcal{A}_{2,10} & \mathcal{A}_{2,11} & \mathcal{A}_{2,12} & \mathcal{A}_{2,13}\\
        \mathcal{A}_{2,1} & \mathcal{A}_{2,2} & \mathcal{A}_{2,3} & \mathcal{A}_{2,4} & \mathcal{A}_{2,5} & i\omega_6/2 & \mathcal{A}_{2,7} & \mathcal{A}_{2,8}& \mathcal{A}_{2,9} & \mathcal{A}_{2,10} & \mathcal{A}_{2,11} & \mathcal{A}_{2,12} & \mathcal{A}_{2,13}\\
        \mathcal{A}_{2,1} & \mathcal{A}_{2,2} & \mathcal{A}_{2,3} & \mathcal{A}_{2,4} & \mathcal{A}_{2,5} & \mathcal{A}_{2,6}& i\omega_7/2 & \mathcal{A}_{2,8}& \mathcal{A}_{2,9} & \mathcal{A}_{2,10} & \mathcal{A}_{2,11} & \mathcal{A}_{2,12} & \mathcal{A}_{2,13}\\
        \mathcal{A}_{2,1} & \mathcal{A}_{2,2} & \mathcal{A}_{2,3} & \mathcal{A}_{2,4} & \mathcal{A}_{2,5} & \mathcal{A}_{2,6}& \mathcal{A}_{2,7} & i\omega_8/2& \mathcal{A}_{2,9} & \mathcal{A}_{2,10} & \mathcal{A}_{2,11} & \mathcal{A}_{2,12} & \mathcal{A}_{2,13}\\
        \mathcal{A}_{2,1} & \mathcal{A}_{2,2} & \mathcal{A}_{2,3} & \mathcal{A}_{2,4} & \mathcal{A}_{2,5} & \mathcal{A}_{2,6}& \mathcal{A}_{2,7} & \mathcal{A}_{2,8}& i\omega_9/2 & \mathcal{A}_{2,10} & \mathcal{A}_{2,11} & \mathcal{A}_{2,12} & \mathcal{A}_{2,13}\\
        \mathcal{A}_{2,1} & \mathcal{A}_{2,2} & \mathcal{A}_{2,3} & \mathcal{A}_{2,4} & \mathcal{A}_{2,5} & \mathcal{A}_{2,6}& \mathcal{A}_{2,7} & \mathcal{A}_{2,8}& \mathcal{A}_{2,9} & i\omega_{10}/2 & \mathcal{A}_{2,11} & \mathcal{A}_{2,12} & \mathcal{A}_{2,13}\\
        \hline
        A_{11,1} & A_{11,2} & A_{11,3} & A_{11,4} & A_{11,5}& A_{11,6}& A_{11,7}& A_{11,8}& A_{11,9}& A_{11,10}& i\omega_{11}/2& A_{11,12}& A_{11,13}\\
        A_{12,1} & A_{12,2} & A_{12,3} & A_{12,4} & A_{12,5}& A_{12,6}& A_{12,7}& A_{12,8}& A_{12,9}& A_{12,10}& A_{12,11}& i\omega_{12}/2& A_{12,13}\\
        A_{13,1} & A_{13,2} & A_{13,3} & A_{13,4} & A_{13,5}& A_{13,6}& A_{13,7}& A_{13,8}& A_{13,9}& A_{13,10}& A_{13,11}& A_{13,12}& i\omega_{13}/2
    \end{pmatrix}\,,
\end{equation*}
    where
    \begin{align*}
    \omega_j = \begin{cases} 
    \text{arbitrary real number} & \text{if } j\in\{1, 5, 11, 12, 13\}\\
    \omega_1 + 2\Imag(\mathcal{A}_{1,j} - \mathcal{A}_{1,1})& \text{if } j\in \{2,3,4\}, \\
    \omega_5 + 2\Imag(\mathcal{A}_{2,j} - \mathcal{A}_{2,5}) & \text{if } j\in \{6,7,8,9,10\}.
\end{cases}
\end{align*}
The effective natural frequencies within each partially integrable part are $\Omega_1 = \omega_1 - 2\Imag(\mathcal{A}_{1,1})$ and $\Omega_2 = \omega_5 - 2\Imag(\mathcal{A}_{2,5})$. Therefore, $c=2$, $p = 3$ meaning that the Kuramoto dynamics can be reduced to $n = 3c + p = 9$ real differential equations. In complex form, they are
\begin{align*}
    \dot{Z}_1 &= F_1(\bm{Z}, \bm{\zeta}, \bm{z}_{\mathcal{P}}) + i\Omega_1 Z_1 - \overline{F_1(\bm{Z}, \bm{\zeta}, \bm{z}_{\mathcal{P}})}Z_1^2\,, \\%&&\forall \gamma \in \{1,...,c\}\\
    \dot{\zeta}_1 &= (i\Omega_1 + F_1(\bm{Z}, \bm{\zeta}, \bm{z}_{\mathcal{P}})\overline{Z}_1 - \overline{F_1(\bm{Z}, \bm{\zeta}, \bm{z}_{\mathcal{P}})}\,Z_1)\,\zeta_1\,,\\
    \dot{Z}_2 &= F_2(\bm{Z}, \bm{\zeta}, \bm{z}_{\mathcal{P}}) + i\Omega_2 Z_2 - \overline{F_2(\bm{Z}, \bm{\zeta}, \bm{z}_{\mathcal{P}})}Z_2^2\,, \\%&&\forall \gamma \in \{1,...,c\}\\
    \dot{\zeta}_2 &= (i\Omega_2 + F_2(\bm{Z}, \bm{\zeta}, \bm{z}_{\mathcal{P}})\overline{Z}_2 - \overline{F_2(\bm{Z}, \bm{\zeta}, \bm{z}_{\mathcal{P}})}\,Z_2)\,\zeta_2\,,\\
    \dot{z}_\ell &= \sum_{k\in \mathcal{P}} (A_{\ell k}z_k - \bar{A}_{\ell k}\overline{z_k}z_\ell^2) + G_{\ell}(\bm{Z}, \bm{\zeta};\bm{w}) - \overline{G_{\ell}(\bm{Z}, \bm{\zeta};\bm{w})}z_\ell^2\,,%\quad&&\forall \ell \in  \mathcal{P}\,,
\end{align*}
where $M_{Z_{\gamma}, \zeta_{\gamma}}(w_j) = \frac{\zeta_{\gamma}w_j + Z_{\gamma}}{1 + \zeta_{\gamma}\,\overline{Z}_{\gamma}w_j}$ and
\begin{align*}
    G_{\ell}(\bm{Z}(t), \bm{\zeta}(t);\bm{w}) = \sum_{\delta =1}^2\sum_{k\in \mathcal{C}_\delta} A_{\ell k}M_{Z_{\delta}(t), \zeta_{\delta}(t)}(w_k)\,,\qquad
    F_{\gamma}(\bm{Z}, \bm{\zeta}, \bm{z}_{\mathcal{P}}) = \sum_{k=11}^{13} \mathcal{A}_{\gamma k}z_k + \sum_{\delta=1}^2 \sum_{k\in \mathcal{C}_\delta} \mathcal{A}_{\gamma k}M_{Z_{\delta}, \zeta_{\delta}}(w_k)\,.
\end{align*}
    Since $n_1 = 4$ and $n_2 = 6$, there are $q = (n_1 - 3) + (n_2 -3) = 4$ constants of motion. The following cross-ratios are functionally independent constants of motion:
    \begin{align*}
        c_{1,2,3,4}(z) &= \frac{(z_{3}-z_{1})(z_{4}-z_{2})}{(z_{3}-z_{2})(z_{4}-z_{1})}
    \end{align*}
    for the part $\mathcal{C}_1$ (i.e., $\mathcal{C}_1$ admits $n_1 -3 = 1$ constant of motion) and
    \begin{align*}
        c_{5,6,7,8}(z) = \frac{(z_{7}-z_{5})(z_{8}-z_{6})}{(z_{7}-z_{6})(z_{8}-z_{5})}\,,\qquad
        c_{6,7,8,9}(z) = \frac{(z_{8}-z_{6})(z_{9}-z_{7})}{(z_{8}-z_{7})(z_{9}-z_{6})}\,,\qquad
        c_{7,8,9,10}(z) = \frac{(z_{9}-z_{7})(z_{10}-z_{8})}{(z_{9}-z_{8})(z_{10}-z_{7})}
    \end{align*}
    for the part $\mathcal{C}_2$ (i.e., $\mathcal{C}_2$ admits $n_2 - 3 = 3$ constants of motion).
    
\end{example}

%\clearpage
%\bibliography{mybib}

\end{document}